\documentclass[12pt,reqno,a4paper]{amsart}
\usepackage{extsizes}
\usepackage{blindtext}
\usepackage{fullpage}
\usepackage{mathtools}
\usepackage{longtable}
\usepackage{amsmath,amssymb,amsthm}
\usepackage{amscd}
\usepackage{bm}
\usepackage{color}
\usepackage{hyperref}
\usepackage{dsfont}
\usepackage{enumerate}
\usepackage{epsfig}
\usepackage{float, graphicx}
\usepackage{booktabs}
\usepackage{latexsym, amsxtra}
\usepackage{mathrsfs}
\usepackage{multicol}
\usepackage[normalem]{ulem}
\usepackage{psfrag}
\usepackage[parfill]{parskip}
\usepackage{stmaryrd}
\usepackage{tikz}
\usepackage[T1]{fontenc}
\usepackage{url}
\usepackage{verbatim}
\usepackage{indentfirst}
\usepackage{tikz-cd}
\usepackage{svg}

\usepackage{mathtools}

\usepackage{comment}

\makeatletter
\def\thm@space@setup{%
	\thm@preskip=2ex \thm@postskip=1.5ex
}
\makeatother

\hypersetup{hidelinks}

\newtheorem{thm}{Theorem~}[section]
\newtheorem{lem}[thm]{Lemma~}

\newtheorem{prop}[thm]{Proposition~}
\newtheorem{ques}[thm]{Question~}
\newtheorem{cor}[thm]{Corollary~}

\theoremstyle{remark}
\newtheorem{rmk}[thm]{Remark~}

\theoremstyle{definition}
\newtheorem{defn}[thm]{Definition~}

\newcommand{\CC}{\mathbb{C}}
\newcommand{\ZZ}{\mathbb{Z}}

\newcommand{\PP}{\mathbb{P}}

\newcommand{\QQ}{\mathbb{Q}}
\newcommand{\DD}{\mathbb{D}}

\newcommand{\Prd}{\mathscr{P}}

\newcommand{\calH}{\mathcal{H}}

\newcommand{\calM}{\mathcal{M}}
\newcommand{\calV}{\mathcal{V}}
\newcommand{\calO}{\mathcal{O}}

\newcommand{\calE}{\mathcal{E}}

\newcommand{\calL}{\mathcal{L}}
\newcommand{\calU}{\mathcal{U}}
\newcommand{\calX}{\mathcal{X}}
\newcommand{\calZ}{\mathcal{Z}}

\newcommand\Aut{\mathrm{Aut}}

\newcommand\Sym{\mathrm{Sym}}
\newcommand\SL{\mathrm{SL}}

\newcommand\Image{\mathrm{Im}}

\newcommand\Id{\mathrm{Id}}
\newcommand\PSL{\mathrm{PSL}}

\newcommand\Or{\mathrm{O}}

\newcommand\Stab{\mathrm{Stab}}

\newcommand{\edim}{\mathrm{edim}}

\newcommand{\Tor}{\mathrm{Tor}}
\newcommand{\rank}{\mathrm{rank}}
\newcommand{\Pic}{\mathrm{Pic}}

\newcommand{\PGL}{\mathrm{PGL}}

\newcommand{\Fix}{\mathrm{Fix}}
\newcommand{\mult}{\mathrm{mult}}

\newcommand{\bs}{\backslash}
\newcommand{\dbs}{\bs\!\! \bs}

\title{Moduli spaces of sextic curves with simple singularities and their compactifications}
\vspace{1cm}
\author{Chenglong Yu, Zhiwei Zheng, Yiming Zhong}
\address{}
\email{}
\date{}

\newcommand{\Addresses}{{
		\bigskip
		\footnotesize
		
		C.~Yu, \textsc{Center for Mathematics and Interdisciplinary Sciences, Fudan University and
Shanghai Institute for Mathematics and Interdisciplinary Sciences (SIMIS), Shanghai, China}\par\nopagebreak
		\textit{E-mail address}: \texttt{yuchenglong@simis.cn}
		
		\medskip
		
		Z.~Zheng, \textsc{Tsinghua University, Beijing, China}\par\nopagebreak
		\textit{E-mail address}: \texttt{zhengzhiwei@mail.tsinghua.edu.cn}

        \medskip
		
		Y.~Zhong, \textsc{Beijing International Center for Mathematical Research, Peking University, Beijing, China}\par\nopagebreak
		\textit{E-mail address}: \texttt{ymzhong@bicmr.pku.edu.cn}
}}

\begin{document}
\begin{abstract}
In this paper, we study moduli spaces of sextic curves with simple singularities. Through period maps of K3 surfaces with ADE singularities, we prove that such moduli spaces admit algebraic open embeddings into arithmetic quotients of type IV domains. For all cases, we prove the identifications of GIT compactifications and Looijenga compactifications. We also describe Picard lattices in an explicit way for many cases. For nodal cases, we prove that the orbifold structures on the two sides of the period map are isomorphic.
\end{abstract}
	
\maketitle
\setcounter{tocdepth}{1}
\tableofcontents

\section{Introduction}
\label{sec: intro}
The study of moduli spaces of singular plane curves has a long history dating back to Severi \cite{severi1921vorlesungen}, where he proved that families of nodal plane curves of fixed degree form smooth varieties of the expected dimension. 
In this paper, we focus on sextic curves. A double cover of $\PP^2$ branched along a smooth sextic curve is a K3 surface of degree two. Thanks to the global Torelli theorem for K3 surfaces \cite{pjateckii-sapiro1971torelli}\cite{burns1975torelli}\cite{looijenga1980torelli}, one can identify the moduli of smooth sextic curves with the complement of two irreducible Heegner divisors in an arithmetic quotient of a type IV domain. In his seminal work \cite{shah1980complete}, Shah described the stable and semistable sextic curves. Based on this, the GIT compactification can be identified with the Looijenga compactification \cite[Theorem 8.6]{looijenga2003compactificationsII} of the moduli of smooth sextic curves.

In this paper, we study moduli of sextic curves with simple singularities. Previously, the first two authors \cite{yu2023moduli} investigated the moduli spaces of nodal sextic curves. 
Their main result identifies the GIT and Looijenga compactifications of these moduli spaces. Additionally, the Picard lattices of the K3 surfaces associated to generic nodal sextic curves with a fixed singular type were determined.
In the present work, we generalize the main results in \cite{yu2023moduli} to sextic curves with arbitrary prescribed simple singularities. 

For a root lattice $R=\oplus_{i=1}^l R_i$ with each $R_i$ irreducible of type ADE, we consider sextic curves with exactly $l$ simple singularities of type $R_1, \cdots, R_l$. These sextic curves form a subvariety of $\PP^{27}=\PP\mathrm{Sym}^6(\CC^3)^\vee$, see \cite[Prop 2.1]{greuel1996equianalytic}. For each irreducible component of this subvariety, we say that the sextic curves in that component have the same singular type, say $T$. We write this component as $\PP\calV_T$, where $\calV_T$ denotes the space of corresponding sextic polynomials.
By results of \cite[Corollary 5.1(b)]{greuel1996equianalytic} and \cite[Theorem 2.5.1]{greuel2021plane}, each $\PP\calV_T$ is smooth of the expected dimension, see \S \ref{subsection: ADE sextic} for the definition of the expected dimension and the precise formula.
For any given singular type $T$, we define the moduli space $\calM_T\coloneqq\SL(3, \CC)\dbs\PP\calV_T$, see \S \ref{subsection: ADE sextic}.


Next, we formulate our main theorem. Fix any singular type $T$ with associated root lattice $R$ and let $Z$ be a sextic curve of type $T$. Let $X$ be the minimal resolution of the double cover of $\PP^2$ branched along $Z$. Then $X$ is a K3 surface. Its Picard lattice $\Pic(X)$ naturally contains a degree two class $H$ (which is the pullback of the hyperplane class of $\PP^2$) and a root lattice $L$ of type $R$. Let $P$ be the primitive hull of $\langle H\rangle\oplus L$ in $\Pic(X)$.
The orthogonal complement $Q$ of $P$ in $\Lambda_X\coloneqq H^2(X,\ZZ)$ is a lattice with signature $(2,20-\rank (P))$, thus defines a type IV bounded symmetric domain $\DD_T\coloneqq\DD(Q)$, where $\DD(Q)$ denotes the type IV domain associated to $Q$ (see Definition \ref{defn: period domain}).
The period of K3 surfaces gives rise to an analytic morphism
\[
\Prd_T\colon \calM_T\to \Gamma_T\bs\DD_T.
\]
Here the monodromy group $\Gamma_T$ is defined in \S\ref{subsection: define period map}. We call $\Prd_T$ an occult period map, following Kudla--Rapoport \cite{kudla2012occult}, to distinguish with the usual period map.

Let $\calH_T$ be the union of reflective hyperplanes for roots $\{ r\in H^\perp \,|\, r^2=-2, r\notin L \}$.
Let $\calH_T^*$ be the sub-arrangement of $\calH_T$ defined by roots with divisibility two. Here a vector $v\in H^\perp$ is called of divisibility $k\in \ZZ^+$ if $(r, H^\perp)=k\ZZ$.  Both $\calH_T$ and $\calH_T^*$ are  $\Gamma_T$-invariant. 
One of the main results of this paper is the following, see Theorem \ref{theorem: global torelli}, Theorem \ref{thm: open}, Theorem \ref{thm: comm diag, compact to compact} and Proposition \ref{prop: orbi str of nodal types}.
\begin{thm}
\label{thm: main in intro}
Given any singular type $T$ of sextic curves with simple singularities. The occult period map $\Prd_T\colon \calM_T \to \Gamma_{T}\bs \DD_T$ is an open embedding whose image is $\Gamma_T\bs(\DD_T-\calH_T)$.
It extends to an isomorphism  
    \[  
    \widehat{\calM}_T \cong \overline{\Gamma_T \bs \DD_T}^{\calH_T^*},  
    \]  
    where $\widehat{\calM}_T$ is the GIT compactification of $\calM_T$ (see \S \ref{subsection: ADE sextic}), and $\overline{\Gamma_T \bs \DD_T}^{\calH_T^*}$ denotes the Looijenga compactification of $\Gamma_T \bs (\DD_T - \calH_T^*)$. 
    Moreover, when $T$ is a nodal singular type, $\Prd_T$ induces an isomorphism of $\calM_T$ and $P\Gamma_T\bs(\DD_T-\calH_T)$ as orbifolds.
\end{thm}
\begin{rmk}
    All root lattices $R$ that arise from some singular type $T$ have been classified by Urabe \cite{urabe1988combinations} and Yang \cite{yang1996sextic}. The maximal rank is $19$.
    By Yang \cite[Thm 2.1]{yang1996sextic}, the number of $R$ with rank $19$ (respectively, $18$, $17$, $16$) is $519$ (respectively, $987$, $975$, $782$).
    These classifications show that there are many singular types to which Theorem~\ref{thm: main in intro} applies.
\end{rmk}

We explain the key ingredients of Theorem \ref{thm: main in intro} and its proof. The first central issue in formulating the theorem is to determine the appropriate subgroup $\Gamma_T \subset \Or(Q)$. It turns out that $\Gamma_T$ should be defined as the group of automorphisms of $Q$ that extend to automorphisms of the K3 lattice preserving $H$ and a chosen base of $L$, see \S\ref{subsection: define period map}. With this definition, the group $\Gamma_T$ is large enough to ensure that the period map $\Prd_T$ well-defined, while still being small enough to allow us to establish the injectivity of $\Prd_T$. 
Next, for each $F\in\calV_T$, we construct an ample class on the K3 surface $X_F$ in a uniform way. 
Combining this construction with our characterization of $\Gamma_T$, the injectivity then follows from the global Torelli theorem for K3 surfaces.

Next, we explain the characterization for the image of the period map, see \S\ref{subsec: image of period map} for more details. Some of the arguments here originate from Urabe \cite{urabe1988combinations} and Degtyarev \cite{degtyarev2008deformations}.
For any $\omega\in \DD_T-\calH_T$, there exists a K3 surface $X_\omega$ by the surjectivity of the period map for K3 surfaces.
We show that $X_\omega$ arises from a sextic curve $Z_\omega$ of singular type $T$, with $\Prd_T(Z_\omega)=[\omega]$ in $\Gamma_T\bs(\DD_T-\calH_T)$.
The proof proceeds in the following two steps. 
First, we construct a nef class $H_\omega$ on $X_\omega$ and show that its complete linear system induces a surjective morphism $\pi_\omega\colon X_\omega\to\PP^2$ of degree two. This morphism contracts exactly the rational curves orthogonal to $H_\omega$ and therefore factors through a finite morphism $p_\omega\colon\widehat{X}_\omega\to\PP^2$, which is a double cover branched along a sextic curve $Z_\omega$. Moreover, the automatic semiregularity of curves on K3 surfaces allows us to construct $\pi_\omega$ in families, from which it follows that $Z_\omega$ has the singular type $T$. 

For the identification between the GIT and Looijenga compactifications, we follow the approach developed in \cite{yu2020fourfolds, yu2023moduli}. The proof can be visualized in Diagram \eqref{diagram: compactifications}.
\begin{equation}
\label{diagram: compactifications}
    \begin{tikzcd}
      \calM_T  \arrow[r,"\Prd_T","\cong"'] \arrow[d,hook] & \Gamma_T\bs(\DD_T-\calH_T) \arrow[d,hook] \\
      \widehat{\calM}_T \arrow[d,"j_T"] & \overline{\Gamma_T\bs\DD_T}^{\calH_T^*} \arrow[d,"\pi_T"] \\
      \overline{\calM} \arrow[r,"\Prd","\cong"'] & \overline{\Gamma_1\bs\DD_1}^{\calH_\infty}
    \end{tikzcd}
\end{equation}
Here $\DD_1$ and $\Gamma_1$ are the type IV domain and monodromy group associated with smooth sextic curves (i.e. $T$ is the smooth type), and $\calH_\infty$ is the Heegner divisor defined by roots of divisibility two, see \S\ref{subsection: Shah, Looijenga}. We have already shown that $\Prd_T$ is an isomorphism between $\calM_T$ and $\Gamma\bs(\DD_T-\calH_T)$. There is a natural map $\pi_T\colon \Gamma_T\bs (\DD_T-\calH_T^*) \to \Gamma_1\bs (\DD_1-\calH_\infty)$. We claim that $\pi_T$ extends to a finite morphism between Looijenga compactifications. 
The key point is the functorial property of Looijenga compactifications $\overline{\Gamma_T\bs\DD_T}^{\calH_T^*}$ and $\overline{\Gamma_1\bs\DD_1}^{\calH_\infty}$, see \cite{looijenga2003compactificationsII} and \cite[Theorem A.13]{yu2020fourfolds}. To verify the hypotheses in \cite[Theorem A.13]{yu2020fourfolds}, we need to check two facts. Firstly, the stabilizer of a generic point of $\DD_T$ in $\Or(\Lambda_X,H)$ is generated by the Weyl group $W(R)$ and $-\Id$. Secondly, we show that $\Gamma_T$ is the normalizer of $W(R)$ in $\Gamma_1$. Then we conclude that $\Prd_T$ extends to an isomorphism between $\widehat{\calM}_T$ and $\overline{\Gamma_T\bs\DD_T}^{\calH_T^*}$, by showing that both vertical maps $j_T$ and $\pi_T$ (see \S \ref{subsec: identification of compactifications}) are normalizations onto their images.
The last assertion of Theorem \ref{thm: main in intro} concerning the orbifold structures is proved based on a lattice-theoretic argument, see Proposition \ref{prop: orbi str of nodal types}



In \S \ref{section: explicit description of lattices}, we further analyze the generic Picard lattice $P$. 
We first consider the lattice $M$ generated by $H$, $L$ and the strict transforms of irreducible components of the sextic curve under $X\to \PP^2$. By the openness of the image of the period map $\Prd_T$, the generic Picard lattice $P$ is the primitive hull of the lattice generated by $H$ and $L$ in $\Lambda_X$, hence it is also the primitive hull of $M$. We prove that $P^\iota=M^\iota$ for the natural involution $\iota$ on $\Lambda_X$, see Proposition \ref{prop: saturation of M_F}.
In \S \ref{sec: orbifold structures}, we give a criterion for when the period map preserves the orbifold structures verify it for all nodal singular types $T$.
In \S \ref{sec: examples and applications}, we apply our results to several interesting examples, including a quintic curve with a line, a quartic curve with two lines, and the classical Zariski pairs. For the singular type given by a quartic curve with two bitangents, we observe an unexpected relationship among three arithmetic quotients of different types (ball type, type IV, and Siegel type).


\textbf{Acknowledgement.}
The first author is supported by the national key research and development program of China (No. 2022YFA1007100) and NSFC 12201337. The second author is partially supported by NSFC 12301058. We thank Samuel Boissi\`{e}re, Bong Lian, Michel Raibaut, Fei Si, Zhiyu Tian and Zheng Zhang for their interest and helpful discussions. 

\section{ADE Sextic Curves and K3 Surfaces}
\label{section: K3}

\subsection{Plane Curves with Simple Singularities}
\label{subsection: ADE sextic}
We always use $Z$ to represent a reduced complex plane curve of positive even degree. We denote by $\widehat{X}$ the double cover of $\PP^2$ branched along $Z$ and let $X$ be the smooth minimal model of $\widehat{X}$. Let $R$ be an irreducible root lattice of ADE type. Suppose $p\in Z$ is a singular point and the corresponding singularity in $\widehat{X}$ is an ADE singularity of type $R$, then one call $p\in Z$ a simple curve singularity of type $R$. 

Suppose $Z$ has only simple singularities.
A plane curve $Z'$ is said to have the same \emph{singular type} as $Z$ if they are equisingular deformation equivalent.
For a general discussion of equisingularity, see \cite{zariski1965equisingularityI,zariski1965equisingularityII,zariski1968equisingularityIII} and \cite{greuel2007singularities}.
We have introduced an alternative definition of singular type in \S \ref{sec: intro}.
The two definitions are equivalent due to \cite[Proposition 2.1]{greuel1996equianalytic}.
Denote by $T$ the singular type of $Z$.
Let $\PP\calV_T$ be the space for curves of a fixed singular type $T$, where $\calV_T$ denotes the cone of polynomials over $\PP\calV_T$. By \cite[Prop 2.1]{greuel1996equianalytic}, $\PP\calV$ is an irreducible quasi-projective variety.

We define the \emph{expected dimension} of $\PP\calV_T$. Let $d=\deg(Z)$. 
Let $x_1,\cdots,x_n$ be the set of singularities of $Z$.
Denote by $R_i$ the corresponding ADE type of $x_i$. The number $\rank(R_i)$ is also called the Milnor number or the Tjurina number for the corresponding surface singularity in other contexts. 
All the ADE singularities are quasi-homogeneous, and for such singularities the corresponding Milnor numbers and Tjurina numbers coincide, see \cite{saito1971quasihomogene}.
The expected dimension of $\PP\calV_T$ is defined to be $\edim(\PP\calV_T)=\binom{d+2}{2}-1-\sum\limits_{i=1}^n \rank(R_i)$. 
When equality holds, we simply say that $\PP\calV_T$ is of the expected dimension.

When $d=6$, the space $\PP\calV_T$ is a smooth irreducible quasi-projective variety of the expected dimension, see \cite[Corollary 5.1(b)]{greuel1996equianalytic}. 

Given a singular type $T$ for ADE sextic curves, let $\overline{\calV}_T$ be the Zariski closure of $\calV_T$ in $\calV=H^0(\PP^2, \calO(6))$. Let $\widehat{\PP\calV}_T$ be the normalization of $\PP\overline{\calV}_T$. Let $\calL$ be the pullback of $\calO_{\PP\calV}(1)$ to $\widehat{\PP\calV}_T$. There is an induced action of $\SL(3, \CC)$ on $(\widehat{\PP\calV}_T, \calL)$, and we consider the GIT quotient
\[
\widehat{\calM}_T=\SL(3, \CC) \dbs (\widehat{\PP\calV}_T, \calL).
\]
Since $\PP\calV_T$ is smooth, it can be naturally identified with its preimage in $\widehat{\PP\calV}_T$. By Shah \cite{shah1980complete}, the points of $\PP\calV_T$ are stable with respect to the action of $\SL(3, \CC)$ on $(\PP\calV, \calO(1))$, hence also stable with respect to the action of $\SL(3, \CC)$ on $(\widehat{\PP\calV}_T, \calL)$.

Let $\calM_T$ be the open subset of $\widehat{\calM}_T$ given by $\PSL(3, \CC)\bs \PP\calV_T$. We call $\calM_T$ the moduli space of sextic curves of type $T$, and call $\widehat{\calM}_T$ the GIT compactification of $\calM_T$. 

Denote by $\overline{\calM}_T$ the closure of $\calM_T$ in $\overline{\calM}$. Let $\pi$ be the GIT quotient $\PP\calV^{ss} \to \overline{\calM}$. An element of $\overline{\calM}$ is a closed semistable $\SL(3, \CC)$-orbit. We claim that $\overline{\calM}_T$ consists of all closed semistable $\SL(3,\CC)$-orbits contained in $\PP\overline{\calV}_T$. 

First, $\PP\overline{\calV}_T \cap \PP\calV^{ss}$ is a closed $\SL(3,\CC)$-invariant subset of $\PP\calV^{ss}$. Since $\pi$ is a GIT quotient, we know $\pi(\PP\overline{\calV}_T \cap \PP\calV^{ss})$ is a closed subset of $\overline{\calM}$. Thus $\pi(\PP\overline{\calV}_T \cap \PP\calV^{ss})\supset\overline{\calM}_T$. One the other hand, $\pi^{-1}(\overline{\calM}_T)$ is a closed subset of $\PP\calV^{ss}$ containing $\PP\calV_T$, hence contains $\PP\overline{\calV}_T \cap \PP\calV^{ss}$. We conclude that $\pi(\PP\overline{\calV}_T \cap \PP\calV^{ss})=\overline{\calM}_T$ and the claim follows.


\begin{prop}
\label{prop: M_T M^bar normalization}
The normalization map $\widehat{\PP\calV}_T \to \PP\overline{\calV}_T$ induces a surjective morphism $j_T\colon \widehat{\calM}_T\to \overline{\calM}_T$, which is also the normalization map.
\end{prop}
\begin{proof}
By \cite[Theorem 1.19]{mumford1994geometric}, $\widehat{\PP\calV}_T^{ss}$ is the preimage of $\PP\calV^{ss}$ in $\widehat{\PP\calV}_T$. Taking the GIT quotients, the normalization map $\widehat{\PP\calV}_T \to \PP\overline{\calV}_T$ canonically induces a generically injective finite surjective morphism $j_T\colon\widehat{\calM}_T \to \overline{\calM}_T$. 
Moreover, $\widehat{\calM}_T$ is normal, hence the morphism $j_T$ is the normalization map.
\end{proof}

\begin{rmk}
In \cite[\S4.2]{yu2023moduli}, the first two authors gave an alternative description for the GIT model in the case when the sextic curves are unions of smooth curves.
We describe some further examples in \S \ref{sec: examples and applications}.
\end{rmk}

\subsection{Resolution of ADE Singularities}
\label{subsection: resolution of ADE}

Consider a smooth projective surface $\widehat{S}$ and a reduced curve $Z\subset \widehat{S}$ with only simple curve singularities.
We assume that there exists a double cover $\widehat{X}$ of $\widehat{S}$ branched along $Z$. Then $\widehat{X}$ has only ADE singularities.
Denote the double cover involution of $\widehat{X}$ by $\iota$. 
Let $p\colon X \to \widehat{X}$ be the minimal resolution of $\widehat{X}$.
The involution $\iota$ naturally extends to $X$.
Denote by $S\coloneqq X/\iota$.
The resolution $p$ is a crepant resolution and is compatible with the involutions on $X$ and $\widehat{X}$. Hence $p$ descends to $S\to \widehat{S}$, which is the minimal blowup such that the strict transform of $Z$ is smooth.


Assume that $Z$ has a singular point $q$ of ADE type $R$, which is necessarily irreducible.
The preimage of $q$ in $X$ is a union of exceptional curves, which we call the resolution graph of $q$. 
The following proposition characterizes the induced $\iota$-action on such a resolution graph. It follows from some direct computations with affine coordinates. For these and other basic facts on ADE singularities, one can see for instance \cite[III, \S 7]{barth2004compact}.
\begin{prop}
\label{prop: iota-action on resolution graph/base}
  The dual graph of the resolution graph of $q$ is the Dynkin diagram of type $R$. 
  The involution $\iota$ naturally acts on the Dynkin diagram of $R$.
  If $R=A_1, D_{2n}, E_7, E_8$, the $\iota$-action on the Dynkin diagram is identity. 
  In other cases, the $\iota$-action on the Dynkin diagram is the unique nontrivial involution of the corresponding Dynkin diagram.
\end{prop}
\begin{rmk}
\label{rmk: iota-action on L(R)}
    We denote by $L(R)$ the associated root lattice of type $R$.
    The resolution of $q$ induces a natural embedding $L(R)\hookrightarrow H^2(X,\ZZ)$.
    Since the exceptional curves define a base of $L(R)$, the induced $\iota$-action on $L(R)$ is given by Propsition \ref{prop: iota-action on resolution graph/base}.
\end{rmk}

\subsection{K3 Surfaces Associated with ADE Sextic Curves}
\label{subsection: ADE K3}
Let $F$ be a homogeneous sextic polynomial of three variables, and $Z(F)$ be the associated sextic curve in $\PP^2$.
We assume that $Z(F)$ has only simple singularities.
Let $\widehat{X}_F$ be the double cover of $\PP^2$ branched along $Z(F)$. Then $K_{\widehat{X}_F}$ is trivial and we call $\widehat{X}_F$ an ADE K3 surface.
The minimal resolution of $\widehat{X}_F$, denoted by $X_F$, is a smooth K3 surface. The quotient surface $S_F\coloneqq X_F/\iota$ admits a canonical birational morphism onto $\PP^2$. The branched locus of $X_F\to S_F$ is the disjoint union of the strict transform of $Z(F)$ and possibly certain exceptional curves.

Fix a $F\in \calV_{T}$ for a given singular type $T$. 
Denote by $p=p_F\colon X_F\to S_F$ the branched double covering.
Let $\iota=\iota_F$ be the involution on $X_F$ induced by the double covering.
Denote by $\Lambda_F$ the lattice $H^2(X_F, \ZZ)$. 
Let $H_F\in H^2(X_F, \ZZ)$ be the pullback of the hyperplane class in $\PP^2$.
Let $l_{T}(R)$ be the number of singularities of type $R$ that appear in the singular K3 surface associated with an element in $\calV_{T}$.
Denote by $L_F=\oplus_R L(R)^{l_{T}(R)}$, where $R$ runs through all irreducible root lattices of ADE type. Let $P_F$ be the primitive hull of $L_F\oplus \langle H_F\rangle$ in $H^2(X_F,\ZZ)$. 
Notice that $P_F$ is a primitive sublattice of the Picard lattice $\Pic(X_F)$. The induced action of $\iota$ on $\Lambda_F$ is still denoted by $\iota$. This action naturally preserves $\Pic(X_F)$, $L_F$ and $P_F$.
For a lattice $M$ equipped with an involution $\iota$, we write $M^{\iota}$ for the $\iota$-invariant sublattice.
By Proposition \ref{prop: iota-action on resolution graph/base} and Remark \ref{rmk: iota-action on L(R)}, we conclude the following:
\begin{prop}
\label{prop: finite index H^2(S) to H oplus L^iota}
    The pullback $p^*\colon H^2(S_F,\ZZ)\to H^2(X_F,\ZZ)$ induces a finite-index extension $H^2(S_F,\ZZ)\hookrightarrow \langle H_F \rangle\oplus L_F^{\iota}$. 
\end{prop}
\begin{cor}
\label{cor: Lambda_F^iota = P_F^iota}
    We have $\Lambda_F^\iota = P_F^\iota$.
    In particular, $\iota$ acts as $-1$ on $P_F^\perp$, where $P_F^\perp$ denotes the orthogonal complement of $P_F$ in $\Lambda_F$. 
\end{cor}
\begin{proof}
    By Lemma \ref{lemma: double cover manifold} (a general topological result, whose proof does not depend on the rest of the paper), $p^*$ induces a finite-index extension $H^2(S,\ZZ)\hookrightarrow \Lambda_F^\iota$.
    Since $\Lambda_F^\iota$ is primitive in $\Lambda_F$, by Proposition \ref{prop: finite index H^2(S) to H oplus L^iota}, $\langle H_F \rangle\oplus L_F^{\iota}$ is a finite-index sublattice in $\Lambda_F^\iota$.
    Since $P_F$ is the primitive hull of $\langle H_F \rangle\oplus L_F$ in $\Lambda_F$, we have $\Lambda_F^\iota\subset P_F$.
    Thus $\Lambda_F^\iota = P_F^\iota$.
\end{proof}

\begin{rmk}
For a nontrivial singular type $T$, the associated K3 surfaces $X_F$ always admit natural elliptic fibration structures.
Explicitly, consider the pencil of lines on $\PP^2$ through any singularity of $Z(F)$. The preimage of a generic such line in $X_F$ is an elliptic curve. Therefore, the K3 surface $X_F$ admits an elliptic fibration. Note that this elliptic fibration may not admit any section (but always admit a multi-section).
\end{rmk}

\section{Occult Period Map and Global Torelli}
\label{section: period map}
In this section, we study the Hodge structures of the associated K3 surfaces $X_F$ for $F\in\calV_T$, and construct the corresponding period map.

\subsection{Period Domain, Monodromy Group and Occult Period Map}
\label{subsection: define period map}
We first define the period domain and monodromy group for a fixed singular type $T$.


Denote by $Q_F=P_F^\perp$ the orthogonal complement of $P_F$ in $\Lambda_F$. Then $P_F$ and $Q_F$ are sublattices of $H^2(X_F,\ZZ)$ with signatures $(1,\sum_R l_{T}(R)\rank(R))$ and $(2,19-\sum_R l_{T}(R)\rank(R))$ respectively. 

\begin{defn}[Period Domain]
\label{defn: period domain}
The space $\PP\{x\in (Q_F)_\CC \big{|} x\cdot x=0, x\cdot \overline{x}>0 \}$ has two connected components that are interchanged by complex conjugation. Denote by $\DD(Q_F)$ the connected component that contains $H^{2,0}(X_F)$. This is the period domain for surfaces $X_F$.
\end{defn}

Let $\Delta_{F}$ be the set of exceptional curves arising from the resolution of singularities, which is a base of $L_F$, see \S \ref{subsection: ADE K3}.
Denote by $\Or(\Lambda_F, \Delta_F, H_F)$ the group of automorphisms $f\in \Or(\Lambda_F)$ such that $f(\Delta_F)=\Delta_F$ and $f(H_F)=H_F$.
An automorphism of $\Lambda_{F}$ which preserves $\Delta_{F}$ and $H_{F}$ necessarily preserves $P_{F}$ and, consequently, induces an automorphism of $Q_{F}$. 
Define $\widetilde\Gamma_{F}$ to be the image of the homomorphism 
\begin{equation*}
    \Or(\Lambda_{F}, \Delta_{F},H_{F})\to \Or(Q_{F}).
\end{equation*}
Let $\Gamma_F$ be the subgroup of $\widetilde\Gamma_{F}$ that preserves $\DD(Q_F)$.
Then $[\widetilde{\Gamma}_F : \Gamma_F] \leq 2$, with equality holds if and only if there exists an element in $\widetilde{\Gamma}_F$ interchanging the two components.
By Corollary \ref{cor: Lambda_F^iota = P_F^iota}, the $\iota_F$ acts as $-1$ on $Q_F$. Thus $-1\in\Gamma_F$ holds for all singular types. 

\begin{prop}
Any element $g\in\Or(Q_{F})$ that acts trivially on the discriminant group $A_{Q_{F}}$ belongs to $\widetilde{\Gamma}_{F}$. In particular, $\Gamma_{F}$ is of finite index in $\Or(Q_{F})$, hence an arithmetic group acting on $\DD(Q_{F})$.
\end{prop}
\begin{proof}
    Let $g\in\Or(Q_{F})$ be an element as in the proposition. By \cite[Corollary 1.5.2]{nikulin1979integer}, there exists an element $\widetilde{g}\in \Or(\Lambda_F)$ such that $\widetilde{g}|_{P_F}=\Id_{P_F}$ and $\widetilde{g}|_{Q_F}=g$. Hence $g$ belongs to $\widetilde{\Gamma}_F$ by definition.
\end{proof}


From now on, we choose a base point $F_0\in \calV_T$, and denote by $\Gamma_T=\Gamma_{F_0}$, $\widetilde{\Gamma}_T=\widetilde{\Gamma}_{F_0}$ and $\DD_T=\DD(Q_{F_0})$. 

For $F\in \calV_T$, choose a path $\gamma$ connecting $F$ and $F_0$. 
The path $\gamma$ induces a diffeomorphism from $X_F$ to $X_{F_0}$ together with an isomorphism 
\[
\gamma^*\colon (\Lambda_F, \Delta_F, H_F, P_F, Q_F)\cong (\Lambda_{F_0}, \Delta_{F_0}, H_{F_0}, P_{F_0}, Q_{F_0}).
\]
The line $\gamma^*(H^{2,0}(X_F))$ represents a point in $\DD_T=\DD(Q_{F_0})$.

Next we define the period map for $X_F$. Choose any two paths $\gamma$, $\gamma'$ in $\calV_T$ connecting $F$ and $F_0$.
Then $\gamma'^*\circ (\gamma^*)^{-1}$ is an automorphism of $(\Lambda_{F_0}, \Delta_{F_0}, H_{F_0}, P_{F_0}, Q_{F_0})$, which induces an element in $\Gamma_T$. Therefore, we have an analytic map
\[
\Prd_T\colon \calV_T\to \Gamma_T\bs \DD_T,
\]
which descends to 
\[
\Prd_T\colon \calM_T\to \Gamma_T\bs \DD_T.
\]
The analytic map $\Prd_T$ is acturally algebraic. This can be deduced from its extension to certain compactifications on both sides, see Theorem \ref{thm: comm diag, compact to compact}. An alternative argument follows the proof of Proposition 2.2.2 in \cite{hassett2000special} using Baily–Borel compactification and the Borel extension theorem.
Following \cite{kudla2012occult}, we call $\Prd_T$ the \emph{occult period map} for sextic curves of type $T$.

\subsection{Global Torelli}
We continue with the notation in \S\ref{subsection: define period map}. Based on the global Torelli theorem for K3 surfaces \cite{burns1975torelli}, we prove the following global Torelli theorem for occult period maps of singular sextic curves.
\begin{thm}
\label{theorem: global torelli}
For any type $T$, the period map $\Prd_T\colon \calM_T \to \Gamma_{T}\bs \DD_T$ is injective.
\end{thm}
\begin{proof}
 Taking $F$, $F'\in \calV_T$ such that $\Prd_T(F)=\Prd_T(F')$, we aim to show that $Z(F)$ is linearly isomorphic to $Z(F')$. We first show that the K3 surfaces $X_F$ and $X_{F'}$ are isomorphic.
Choose two paths $\gamma_F$ and $\gamma_{F'}$ connecting $F$ and $F'$ to $F_0$.
There exists an element $g\in \Gamma_{T}$ such that $g\cdot\gamma_F^*(H^{2,0}(X_F))=\gamma_{F'}^*(H^{2,0}(X_{F'}))$.
By the definition of $\Gamma_{T}$, there exists $\widetilde{g}\in \Or(\Lambda_{F_0}, \Delta_{F_0},H_{F_0})$ whose restriction to $Q_{F_0}$ is $g$.
Then the isomorphism 
\[
\phi=(\gamma_{F'}^*)^{-1}\circ \widetilde{g}\circ \gamma_F^* \colon (\Lambda_F, \Delta_F, H_F)\cong (\Lambda_{F'}, \Delta_{F'}, H_{F'})
\]  
maps $H^{2,0}(X_F)$ to $H^{2,0}(X_{F'})$.

Choose an element $\beta_F$ in the Weyl chamber of $L_F$ relative to $\Delta_F$ and let $\beta_{F'}=\phi(\beta_F)$. For a smooth rational curve $D$ in $X_F$, we have either $[D]\cdot H_F$ a positive integer or $[D]\in \Delta_F$. The latter case implies $[D]\cdot \beta_F>0$. Thus there exists a positive integer $c$ such that $\alpha_F=\beta_F+cH_F$ has a positive intersection with any such $[D]$. By Nakai--Moishezon criterion for K3 surfaces, $\alpha_F$ is ample. If $c$ is large enough, we may ask $\alpha_{F'}=\beta_{F'}+cH_{F'}$ also ample. 
Then the isomorphism $\phi\colon \Lambda_F\to \Lambda_{F'}$ is a Hodge isometry that identifies the two ample classes $\alpha_F$ and $\alpha_{F'}$. By the global Torelli theorem for K3 surfaces \cite{burns1975torelli}, there exists a unique isomorphism $\Phi\colon X_F\to X_{F'}$ induces $\phi$. 

We finally aim to construct an isomorphism of $Z(F)$ to $Z(F')$.
Notice that $\phi$ is compatible with the deck transformations $\iota_F$ and $\iota_{F'}$. By the faithfulness of the induced action of automorphisms of K3 surfaces on the middle cohomology \cite{looijenga1980torelli}, $\Phi$ is compatible with $\iota_F$ and $\iota_{F'}$.
Hence $\Phi$ induces isomorphisms $\Fix(\iota_F)\cong \Fix(\iota_{F'})$ and $S_F\cong S_{F'}$. Furthermore, $\Phi$ induces isomorphism between linear systems of $H_F$ and $H_{F'}$. Therefore, $\Phi$ descends to a linear isomorphism $(\PP^2, Z(F))\cong (\PP^2, Z(F'))$. The injectivity of $\Prd_T\colon\calM_T\to \Gamma_T\bs \DD_T$ follows.
\end{proof}

\subsection{Image of the Period Map $\Prd_T$}
\label{subsec: image of period map}
Next we characterize the image of the period map $\Prd_T$. Many of the results in this section are essentially a reinterpretation of \cite[Theorem 1.16]{urabe1988combinations} and \cite[Theorem 3.4.1]{degtyarev2008deformations}. Our description focuses on the level of moduli spaces and arithmetic quotients.


Since we fix the base point $F_0$, we omit the subscript $F_0$ for $\Lambda_{F_0},H_{F_0},\Delta_{F_0},L_{F_0},P_{F_0},Q_{F_0}$ throughout this section.
Let $\calH_T$ be the union of $r^{\perp}\subset \DD_T$ with $r$ running through all roots $r$ such that $r\perp H$ and $r\notin L$.


\begin{thm}
\label{thm: open}
For any type $T$, the period map $\Prd_T\colon\calM_T\to\Gamma_T\bs\DD_T$ induces an isomorphism $\calM_T \cong \Gamma_T\bs (\DD_T-\calH_T)$.
\end{thm}

The orthogonal complement $H^\perp$ of $H$ in $\Lambda$ has signature $(2,19)$ and defines a $19$-dimensional type IV domain $\DD(H^\perp)$.
There is a natural inclusion $\DD_T \subset \DD(H^\perp)$.

\begin{lem}
\label{lem: no calH'_T}
    For any $\omega\in \DD_T-\calH_T$, there does not exist $u\in \Lambda$ such that $u\cdot\omega=0$, $u^2=0$ and $u\cdot H=1$.
\end{lem}
\begin{proof}
Suppose there exists an element $u$ satisfying the requirement.
We observe that $(2u-H)^2=-2$ and $(2u-H)\cdot H=0$, thus $2u-H$ is a root in $\langle H,\omega \rangle^\perp_{\Lambda}$.
    By the choice of $\omega$, we have $2u-H\in L$, then $u\in P \subset \Pic(X_{F_0})$.
Hence $\Pic(X_{F_0})$ contains $u$ such that $u^2=0$ and $u\cdot H=1$.
    By \cite[Proposition 1.7]{urabe1988combinations}, this contradicts the fact that a line bundle on $X_{F_0}$ representing $H$ defines a surjective morphism $X_{F_0}\to \PP^2$ of degree two.
\end{proof}

Denote by $\omega_0\coloneqq H^{2,0}(X_{F_0})$.
Take any $\omega\in\DD_T-\calH_T$.
By surjectivity of the period map for marked complex K3 surfaces, there exists a K3 surface $X_\omega$ with a marking $\phi_\omega\colon H^2(X_\omega,\ZZ)\to \Lambda$ such that $\phi_\omega(H^{2,0}(X_\omega))=\omega$.
From now on, we fix $\omega$ throughout the proof.
The following lemma is essentially proved by Urabe \cite[Theorem 1.16]{urabe1988combinations}:
\begin{lem}
\label{lem: X_omega determines Z(F)}
    The K3 surface $X_\omega$ is the minimal resolution of the double cover of $\PP^2$ branched along a singular sextic curve $Z_\omega$.
    The singular sets of $Z_\omega$ is determined by $\Delta$.
\end{lem}
\begin{proof}
    By \cite[Chapter 8, \S 2, Corollary 2.9]{huybrechts2016lectures}, there exist rational curves $C_1,\cdots,C_n$ on $X_\omega$ such that 
    \begin{equation*}
        s_{[C_1]}\circ\cdots\circ s_{[C_n]}(\phi_\omega^{-1}H)
    \end{equation*}
    is a nef class with self-intersection number $2$, where $[C_i]$ is a root in $\Pic(X_\omega)$ and $s_{[C_i]}$ denotes the refleciton with respect to the root $[C_i]$.
    Since $s_{[C_i]}$ acts trivially on $A_{\Pic(X_\omega)}$, it naturally extends to an isometry $\widetilde{s_{[C_i]}}$ of $H^2(X_\omega,\ZZ)$ that acts trivally on $\Pic(X_\omega)^\perp$. 
    Denote by 
    \begin{equation*}
        \phi'_\omega\coloneqq \phi_\omega\circ \widetilde{s_{[C_n]}}\circ\cdots\circ \widetilde{s_{[C_1]}}.
    \end{equation*}
    The new marking $\phi'_\omega$ of $X_\omega$ satisfies that $\phi'^{-1}_\omega H$ is nef and $\phi'_\omega(H^{2,0}(X_\omega))=\omega$.
    By \cite[Proposition 1.7]{urabe1988combinations} and Lemma \ref{lem: no calH'_T}, the complete linear system of $\phi'^{-1}_\omega H$ defines a surjective morphism $\pi_\omega\colon X_\omega\to \PP^2$ of degree $2$.

    By the choice of $\omega$, the root lattice of $\{x\in H^\perp \big{|} x \perp \omega \}$ equals to $L$.  
    (Note that $\omega$ not only determines the root lattice $L$ up to isomorphism, but also determines an embedding of $L$ into $H^\perp\subset\Lambda$.)
    We first show that all rational curves contracted by $\pi_\omega$ form a base of $\phi'^{-1}_\omega L$.
    Let $C$ be a connected curve in $X$ that is contracted by $\pi_\omega$.
    This is equivalent to say that $C\cdot \phi'^{-1}_\omega H = 0$.
    By Hodge index theorem, $C^2<0$, hence $C^2=-2$. This implies $[C]\in \phi'^{-1}_\omega L$.
    Moreover, $C$ is a connected $\PP^1$-graph whose dual graph is a Dynkin diagram of certain ADE type.
    Conversely, any rational curve in $X$ that defines a root in $L$ is contracted by $\pi_\omega$, since it is orthogonal to $\phi'^{-1}_\omega H$.
    The cohomology classes of all these rational curves generate the root lattice $\phi'^{-1}_\omega L$.
    Hence they form a base of $\phi'^{-1}_\omega L$.

    This concludes that $\pi_\omega$ factor through a birational morphism $X_\omega \to \widehat{X}_\omega$ that contracts all rational curves orthogonal to $\phi'^{-1}_\omega H$, where these rational curves form the unique effective base of $\phi'^{-1}_\omega L$. Thus the base is exactly $\phi'^{-1}_\omega \Delta$.
    The normal surface $\widehat{X}_\omega$ is a singular K3 surface with ADE singularities, where the singularities are the image of rational curves in $\phi'^{-1}_\omega \Delta$.
    Since $\pi_\omega$ is surjective, the induced map $p_\omega\colon \widehat{X}_\omega \to \PP^2$ is also surjective.
    Hence $p_\omega$ is a branched double covering.
    By \cite[III, \S 7, Theorem 7.2]{barth2004compact}, the branched locus of $p_\omega$ is a sextic curve $Z_\omega$ on $\PP^2$.
    Therefore, the map $\pi_\omega$ is a double cover branched along a sextic curve composed of the contraction of $(-2)$-curves on $X_\omega$, and the singularities of the sextic curve $Z_\omega$ are the images of $(-2)$-curves in $\phi'^{-1}_\omega \Delta$.
\end{proof}

The following lemma is essentially proved by Degtyarev \cite[Theorem 3.4.1]{degtyarev2008deformations}. We provide an alternative proof for the reader's convenience.
\begin{lem}
\label{lem: Z_omega has type T}
    The singular sextic curve $Z_\omega$ is equisingular deformation equivalent to $Z(F_0)$, namely, $Z_\omega\in \PP\calV_T$.
\end{lem}
\begin{proof}
    We choose arbitrarily a path $\gamma_\omega$ in $\DD_T-\calH_T$ connecting $\omega$ and $\omega_0$.
    For any $\omega'\in\gamma_\omega$, by the proof of Lemma \ref{lem: X_omega determines Z(F)}, there exists a marking $\phi_{\omega'}$ of $X_{\omega'}$ such that $\phi_{\omega'}^{-1} H$ is nef.
    By \cite[Proposition 1.7]{urabe1988combinations} and \cite[Chapter 2, \S 3, Corollary 3.14]{huybrechts2016lectures}, any line bundle $\calO(C_{\omega'})$ representing the cohomology class $\phi_{\omega'}^{-1} H$ is base point free, where $C_{\omega'}$ is a nef divisor in $X_{\omega'}$.
    By Bertini's theorem, a generic member in the complete linear system $|C_{\omega'}|$ is a smooth genus two curve.
    We may assume that $C_{\omega'}$ is a smooth genus two curve.
    Denote by $\calE_{\omega'}\coloneqq\{E_r \,|\, r\in \phi_{\omega'}^{-1} \Delta\}$ be the set of exceptional curves in $X_{\omega'}$ representing $\phi_{\omega'}^{-1} \Delta$.

    The universal deformation of $X_{\omega'}$ in the moduli space of complex K3 surfaces induces a family of K3 surfaces over a contractible complex neighborhood $D$ of $\omega'$ in $\DD_T-\calH_T$, by the local Torelli theorem.
    The marking $\phi_{\omega'}$ of $X_{\omega'}$ canonically induces markings for all fibers $H^2(X_\eta,\ZZ)\cong H^2(X_{\omega'},\ZZ)\cong\Lambda$, $\eta\in D$.
    Since $\eta^\perp_\Lambda$ contains $H$ and $\Delta$, $\phi_{\omega'}^{-1} H$ and $\phi_{\omega'}^{-1} \Delta$ lifts to algebraic classes over $D$.
    Due to \cite[Page 363]{nishinou2024deformation}, any curve on a K3 surface is semiregular.
    By \cite[Theorem 7.1]{bloch1972semi-regularity} and \cite[Theorem 1.5(i)]{artin1968solutions}, there exists a complex neighborhood $D_{\omega'}\subset D$ of $\omega'$ in $\DD_T-\calH_T$ such that for any $\eta\in D_{\omega'}$, $C_{\omega'}$ and $\calE_{\omega'}$ deform to smooth algebraic curves $C_{\eta}$ and $\calE_{\eta}$ on $X_\eta$, where $[C_{\eta}]=\phi_\eta^{-1} H$.

    Denote by $\calX$ the family of K3 surfaces over $D_{\omega'}$. 
    For any $\eta\in D_{\omega'}$, the complete linear system of $\calO(C_\eta)$ defines a morphism $X_\eta\to \PP H^0(\calO(C_\eta))^\vee\cong\PP^2$ of degree two.
    There exists a line bundle $\calL$ on $\calX$ such that $\calL_\eta=\calO(C_\eta)$ for any $\eta\in D_{\omega'}$.
    Denote by $\lambda\colon \calL\to\calX$ the natural projection.
    We have a relative projective morphism $\calX\to \PP R^0\lambda_*\calL^\vee$ over $D_{\omega'}$ defined by the relative linear system $|\calL|$.
    Note that $\PP R^0\lambda_*\calL^\vee$ is a trivial family of $\PP^2$ over $D_{\omega'}$.
    We can go through the process of Lemma \ref{lem: X_omega determines Z(F)} relatively over $D_{\omega'}$, then obtain a family of singular sextic curves $\calZ\subset \PP R^0\lambda_*\calL^\vee$ sharing the same type of singularities.
    Then any two sextic curves in this family are equisingular deformaiton equivalent.

    The path $\gamma_\omega$ is covered by finite complex neighborhoods in $\DD_T-\calH_T$, and this concludes that $(Z_\omega,\PP^2)$ is equisingular deformation equivalent to $(Z(F_0),\PP^2)$.
\end{proof}
Then the sextic curve $Z_\omega$ can be written as $Z(F)$ for some sextic polynomial $F\in\calV_T$.
The associated K3 surface $X_\omega$ is the same as $X_F$, which we construct in \S \ref{subsection: ADE K3}.

\begin{lem}
\label{lem: Prd_T coincides with global period map}
    The occult period map $\Prd_T$ maps $Z(F)$ to $[\omega]$ in $\Gamma_T\bs (\DD_T-\calH_T)$.
\end{lem}
\begin{proof}
    We choose arbitrarily a path $\gamma_F$ connecting $Z(F)$ and $Z(F_0)$ in $\PP\calV_T$. 
    The path $\gamma_F$ canonically induces a marking $\gamma_F^*\colon H^2(X_F,\ZZ)\to\Lambda$.
    The induced cohomology class $H_F\coloneqq (\gamma_F^*)^{-1} H\in H^2(X_F,\ZZ)$ is independent of the choice of the path $\gamma_F$, which is equal to the pullback of the hyperplane class in $\PP^2$.
    By the proof of Lemma \ref{lem: X_omega determines Z(F)}, there exists a marking $\phi_\omega$ of $X_\omega=X_F$ such that $\phi_\omega^{-1} H = H_F$ and $\phi_\omega^{-1} \Delta = \Delta_F$.
    Consider $g\coloneqq \gamma_F^*\circ \phi_\omega^{-1}\in \Or(\Lambda)$, it preserves $H$ and $\Delta$. Thus the restriction $g|_Q$ belongs to $\widetilde{\Gamma}_T$.
    It is direct $g\cdot\omega = \gamma_F^* H^{2,0}(X_F)$. 
    Both $\omega$ and $\gamma_F^* H^{2,0}(X_F)$ lies in $\DD_T$, hence $g$ sends $\DD_T$ to $\DD_T$. It follows that $g|_Q\in \Gamma_T$. We conclude that $\Prd_T$ maps $Z(F)$ to $[g\cdot\omega]=[\omega]$ in $\Gamma_T\bs (\DD_T-\calH_T)$.
\end{proof}

\begin{proof}[Proof of Theorem \ref{thm: open}] 
    For any $F\in\calV_T$, the natural degree two morphism $X_F\to\PP^2$ is induced by the complete linear system of $H_F$.
    The set of curves contracted by $|H_F|$ is exactly $\Delta_F$.
    Thus the root lattice of $\langle H_F,H^{2,0}(X_F) \rangle^\perp_{\Lambda_F}$ is generated by $\Delta_F$, hence equals $L_F$.
    This implies that $\Prd(\calM_T)$ lies in $\Gamma_T\bs(\DD_T-\calH_T)$.
    The opposite direction follows directly from Lemma \ref{lem: X_omega determines Z(F)}, Lemma \ref{lem: Z_omega has type T} and Lemma \ref{lem: Prd_T coincides with global period map}.
\end{proof}


\begin{cor}
\label{cor: generic Picard}
For any generic $F\in\calV_T$ we have $\Pic(X_F)=P_F$.
\end{cor}
\begin{proof}
By Theorem \ref{thm: open}, the period map $\Prd_T \colon\calM_T\to \Gamma_T\bs \DD_T$ is open. For a generic $F\in \calV_T$, the period of $F$ is a general point in $\DD_T$. Thus the Picard lattice of $X_F$ is $P_F$.
\end{proof}

\begin{rmk}
    The hyperplane arrangement $\calH_T$ usually has infinite irreducible components, but the quotient $\Gamma_T\bs\calH_T$ is, in fact, a divisor of $\Gamma_T\bs\DD_T$.
    This can be concluded from \cite{baily1966compactification}, \cite{borel1972metric} and Theorem \ref{thm: open}.
\end{rmk}

\begin{rmk}
Many moduli spaces of sextic curves with specified singularities and group actions are closely related to Deligne--Mostow theory \cite{deligne1986monodromy}. For example, Dolgachev, van Geemen and Kond\=o \cite{dolgachev2005complex} studied moduli space of pairs consisting of a smooth cubic surface and a line on it. Such pairs are naturally related to pairs consisting of a smooth quintic curve and a line that admits a $\mu_3$-symmetry. The moduli space corresponds to
the case $({1\over 6}, {1\over 6}, {1\over 3}, {1\over 3}, {1\over 3}, {1\over 3}, {1\over 3})$ in \cite{mostow1988discontinuous}. In general, the cases of nodal sextic curves with symmetries have been studied in \cite[\S 5]{yu2023moduli}. In \cite{zheng2024complex} the last two authors studied irreducible sextic curves with a $D_4$-singularity and $\mu_3$-symmetry and the moduli is related to case $({1\over 6}, {1\over 6}, {1\over 6}, {1\over 6}, {1\over 6}, {1\over 6}, {1\over 3}, {1\over 3}, {1\over 3})$ in \cite{mostow1988discontinuous}. 
\end{rmk}


\section{GIT Compactifications and Arithmetic Compactifications}
\label{section: compactification}

\subsection{K3 Surfaces of Degree $2$}
\label{subsection: Shah, Looijenga}
We first review the classical results of moduli space of sextic curves via Hodge theory and compactifications. 
See \cite{shah1980complete} and Looijenga.

Let $V$ be a $3$-dimensional complex vector space.
A (plane) sextic curve is the zero locus of an element of $\Sym^6 V^\vee$ in $\PP V$.
By definition, $\PP\Sym^6 V^\vee$ is the space of sextic curves. 
Let $\PP(\Sym^6 V^\vee)^{\mathrm{sm}}$ be the subset of $\PP\Sym^6 V^\vee$ consisting of smooth sextic curves.
We define $\calM$ as the GIT quotient $\SL(V)\dbs\PP(\Sym^6 V^\vee)^{\mathrm{sm}}$.
Denote by $\overline{\calM}$ the GIT compactification of $\calM$, and $\calM^{\mathrm{ADE}}\subset \overline{\calM}$ the subspace consisting of sextic curves with at worst simple singularities.

We briefly recall the definition of period domain and arithmetic group of the moduli space of polarized K3 surfaces of degree $2d$.
Let $\Lambda\coloneqq U^3\oplus E_8^2$ be the K3 lattice.
Let $l\coloneqq e_1+df_1$, where $e_1,f_1$ is the standard basis of the first copy of $U$, and $l^2=2d$.
Denote by $\Lambda_d$ for $l^{\perp}$ in $\Lambda_{K3}$.
The subset of $\PP(\Lambda_{d,\CC})$ defined by $z\cdot z=0$ and $z\cdot\bar{z}>0$ has two connected components that are interchanged by complex conjugation. Denote by $\DD_d$ and $\overline{\DD}_d$ the two components.
Define a subgroup of $\Or(\Lambda)$
\begin{equation*}
    \widetilde\Gamma_d\coloneqq \{ g|_{\Lambda_d} \,\big|\, g\in \Or(\Lambda),\, g(e_1+df_1)=e_1+df_1 \}.
\end{equation*}
Equivalently, $\widetilde\Gamma_d$ is isomorphic to the subgroup of $\Or(\Lambda_d)$ which consists of elements that act trivially on $A_{\Lambda_d}$.
Moreover, $\widetilde\Gamma_d$ is an arithmetic subgroup of $\Or(\Lambda_d)$.
Let $\Gamma_d$ denote the index two subgroup of $\widetilde\Gamma_d$ that preserves $\DD_d$.
By Baily-Borel, the arithmetic quotient $\Gamma_d\bs\DD_d$ is a normal quasi-projective variety.


For a $Z\in \PP(\Sym^6 V^\vee)^{\mathrm{sm}}$, the double cover of $\PP V$ branched along $Z$ is a smooth K3 surface $X$.
Denote by $H\in H^2(X,\ZZ)$ the pullback of the hyperplane class of $\PP V$, which is a degree $2$ polarization of $X$.
Then the period map for polarized K3 surfaces of degree $2$ (for example, see \cite[Chapter 6, \S 4.1]{huybrechts2016lectures}) induces a period map $\Prd\colon \PP(\Sym^6 V^\vee)^{\mathrm{sm}}\to \Gamma_1\bs\DD_1$.
The period map $\Prd$ is $\SL(V)$-equivariant, hence descends to 
\begin{equation*}
    \Prd\colon \calM\to \Gamma_1\bs\DD_1.
\end{equation*}

We have two $\Gamma_1$-invariant hyperplane arrangements $\calH_\Delta$ and $\calH_\infty$ in $\DD_1$, which correspond to two $\Gamma_1$-orbits of roots in $\Lambda_1$.
The roots defining $\calH_\Delta$ (respectively, $\calH_\infty$) has divisibility $1$ (respectively, $2$).
We have the Looijenga compactification of $\Gamma_1\bs(\DD_1-\calH_\infty)$, denoted by $\overline{\Gamma_1\bs\DD_1}^{\calH_\infty}$, see \cite{looijenga2003compactificationsII}.
From \cite{shah1980complete} and \cite{looijenga2003compactificationsII}, we have
\begin{thm}[Shah, Looijenga]
\label{thm: Shah--Looijenga}
    The period map $\Prd$ is an algebraic open embedding whose image is $\Gamma_1\bs(\DD_1-\calH_\Delta-\calH_\infty)$.
    Moreover, $\Prd$ extends to an isomorphism $\calM^{\mathrm{ADE}}\cong\Gamma_1\bs(\DD_1-\calH_\infty)$, and further to $\overline{\calM}\cong\overline{\Gamma_1\bs\DD_1}^{\calH_\infty}$.
\end{thm}

\subsection{Identification of GIT Compactification and Looijenga Compactification}
\label{subsec: identification of compactifications}

In this section, we describe the relationship between GIT compactifications and Looijenga compactifications.
We fix a singular type $T$.



\subsubsection{$\Gamma_T\bs\DD_T$ and $\Gamma_1\bs\DD_1$}
We next describe the relation between $\Gamma_T\bs\DD_T$ and $\Gamma_1\bs\DD_1$.

Let $X$ be a K3 surface associated with a sextic curve of type $T$. 
There is the tuple $(H^2(X,\ZZ),H,\Delta)$ associated with $X$.
Denote by $\Lambda_X\coloneqq H^2(X,\ZZ)$.
We fix a marking $\phi\colon \Lambda_X\to\Lambda$ such that $\phi(H)=l\coloneqq e_1+f_1$.
Then the type IV domain $\DD_H$ associated with $H^\perp\subset \Lambda_X$ is identified with $\DD_1$ by $\phi$.

Let $W$ be the Weyl group of the root lattice $L=\oplus_{i=1}^k R_i$. Then $W$ acts trivially on $A_L$.
Thus the action of $W$ on $L$ extends to $\Lambda_X$ such that the action on $L^\perp$ is identity.
It defines a natural inclusion $W\subset\Or(\Lambda_X,H)$.
The $W$-action on $\Lambda_X$ induces an action on $\Lambda$ by the conjugation of $\phi$, which is by definition $\phi$-equivariant.
The group $\Or(\Lambda_X,H)$ is also identified with $\widetilde\Gamma_1$ via $\Or(\Lambda_X,H) = \phi^{-1}\widetilde\Gamma_1\phi$.


The period domain $\DD_T$ is defined to be the type IV domain associated with $P^\perp$, where $P^\perp$ is the $W$-invariant part of $H^{\perp}$.
Namely, $\DD_T$ is the type IV domain associated with the character subspace of $H^\perp\otimes\CC$ with respect to the trivial character of $W$.

Define $\widetilde\Gamma_{P^\perp}\coloneqq\{ g\in\Or(\Lambda_X,H) \,\big|\, g(P^\perp)=P^\perp \}$ and $\Gamma_{P^\perp}\coloneqq\{ g\in\Or(\Lambda_X,H) \,\big|\, g(\DD_T)=\DD_T \}$.
Note that $[\widetilde{\Gamma}_T:\Gamma_T]\le 2$.
Let $\widetilde{W}\subset\Or(\Lambda_X,H)$ be the group generated by $W$ and $-\Id$.
\begin{lem}
\label{lem: Stab(x) = W}
    The group $\widetilde{W}$ is the stabilizer of a generic point of $\DD_T$ in $\Or(\Lambda_X,H)$.
\end{lem}
\begin{proof}
    Let $x\in\DD_T$ be a point such that $x^\perp\cap\Lambda_X=P$.
    The stabilizer of $x$ in $\Or(\Lambda_X,H)$, denoted by $\Stab(x)$, is contained in $\Gamma_{P^\perp}$.
    Consider the map $\Gamma_{P^\perp}\to\Or(P^\perp)$ defined by restriction.
    Denote by $K$ its kernel, and let $\widetilde{K}\subset\Gamma_{P^\perp}$ be the group generated by $K$ and $-\Id$.
    We claim that $\Stab(x)=\widetilde{K}$. 
    It is clear that $\widetilde{K}\subset\Stab(x)$.
    For any $g\in\Or(\Lambda_X,H)$ such that $g v_x=\lambda v_x$ for some $\lambda\in\CC^*$, where $v_x\in P^\perp\otimes\CC$ satisfies $[v_x]=x\in\DD_T$.
    For any $v\in P^\perp$, we have
    \begin{equation*}
        v\cdot v_x = g v\cdot g v_x = g v\cdot \lambda v_x = \lambda g v\cdot v_x,
    \end{equation*}
    then by $v_x^\perp\cap\Lambda_X=P$, $g v=\lambda^{-1}v$.
    Since $\lambda=\pm1$, we have $\Stab(x)=\widetilde{K}$.
    
    The Weyl group $W$ extends on $P^\perp$ by trivial action, hence $W\subset K$.
    For any $g\in K$, $g$ acts on $A_P$ trivially.
    Thus $g$ acts on $A_L$ trivially.
    From \cite[PLATE I--X]{bourbaki2002lie4-6}, $\Aut(R)/W(R)$ acts faithfully on $A_R$ for any irreducible finitely generated positive definite root lattice $R$.
    Hence $g|_L\in\Or(L)$ lies in $W$.
    Thus $K\subset W$. It follows that $\widetilde{K}=\widetilde{W}$.
\end{proof}

Let $\widetilde\Gamma_W$ (respectively, $\Gamma_W$) be the normalizer of $W$ in $\Or(\Lambda_X,H)$ (respectively, $\phi^{-1}\Gamma_1\phi$). 
Notice that $\widetilde\Gamma_W$ (respectively, $\Gamma_W$) is also the normalizer of $\widetilde{W}$ in $\Or(\Lambda_X,H)$ (respectively, $\phi^{-1}\Gamma_1\phi$). 
Note that $[\widetilde{\Gamma}_W:\Gamma_W]\le 2$.
By definition, $\widetilde\Gamma_W$ preserves the $W$-invariant part of $\Lambda_X$, which is $P^\perp$.

\begin{lem}
\label{lem: Gamma_P^perp = Gamma_W}
    We have $\Gamma_{P^\perp}=\Gamma_W$.
\end{lem}
\begin{proof}
    We first show that $\widetilde\Gamma_{P^\perp}=\widetilde\Gamma_W$.
    It is clear that $\widetilde\Gamma_{P^\perp} \supset \widetilde\Gamma_W$.
    For any $g\in\Or(\Lambda_X,H)$ such that $g(P^\perp)=P^\perp$, we have $g(L)=L$, namely $g|_L\in\Or(L)$.
    Since $W$ is a normal subgroup of $\Or(L)$, we have $g^{-1}Wg=W$.
    Thus $\widetilde\Gamma_{P^\perp} \subset \widetilde\Gamma_W$.
    For any $g\in\widetilde\Gamma_W$, $\phi g\phi^{-1}$ preserves $\DD_1$ is equivalent to that $g$ preserves $\DD_T$.
    It follows $\Gamma_{P^\perp}=\Gamma_W$.
\end{proof}

\begin{rmk}
    Due to \cite[Lemma A.4]{yu2020fourfolds}, Lemma \ref{lem: Gamma_P^perp = Gamma_W} is a corollary of Lemma \ref{lem: Stab(x) = W}.
\end{rmk}



Recall that $\widetilde{\Gamma}_T$ denotes the image of $\Or(\Lambda_X,\Delta,H) \to \Or(P^\perp)$, and $\Gamma_T$ is the subgroup of $\widetilde{\Gamma}_T$ that preserves $\DD_T$ (thus $[\widetilde{\Gamma}_T:\Gamma_T]\le 2$).
The following result gives an alternative description of the arithmetic group $\Gamma_T$.
\begin{lem}
\label{lem: Gamma_W|_P^perp = Gamma_T}
    The arithmetic group $\Gamma_W|_{P^\perp}$ coincides with $\Gamma_T$.
\end{lem}
\begin{proof}
    We only need to verify $\widetilde{\Gamma}_W|_{P^\perp}=\widetilde{\Gamma}_T$, namely
    \begin{equation*}
        \Image( \{ g\in \Or(\Lambda_X,H) \big| g^{-1}Wg=W \} \to \Or(P^\perp) ) = \Image ( \Or(\Lambda_X,\Delta,H) \to \Or(P^\perp) ).
    \end{equation*}
    We first verify $\widetilde\Gamma_W|_{P^\perp} \supset \widetilde\Gamma_T$.
    An element in $\widetilde\Gamma_T$ can be expressed as $g|_{P^\perp}$ for some $g\in \Or(\Lambda_X,\Delta,H)$.
    The Weyl group $W$ is generated by the set of reflections $\{ s_\alpha \,|\, \alpha\in \Delta \}$.
    For each $\alpha\in \Delta$, we have $g^{-1} s_\alpha g = s_{g^{-1}\alpha}$.
    By definition of $g$, $g^{-1}\alpha$ is also an element in $\Delta$, hence $g^{-1} s_\alpha g \in W$.
    Thus we have $\widetilde\Gamma_W|_{P^\perp} \supset \widetilde\Gamma_T$.

    We next verify $\widetilde\Gamma_W|_{P^\perp} \subset \widetilde\Gamma_T$.
    An element $g$ in $\Or(\Lambda_X,H)$ satisfying $g^{-1}Wg=W$ preserves $P$ and $P^\perp$.
    Since $g$ preserves $H$, $g$ also preserves $L$, namely, $g|_{L}\in \Or(L)$.
    Then $g(\Delta)$ is also a base of $L$.
    The Weyl group $W$ acts on the set of bases of $L$ transitively.
    Hence there exists a $\sigma\in W$ such that $\sigma (g(\Delta)) = \Delta$, i.e. $(\sigma g) (\Delta) = \Delta$.
    By definition, $\sigma g\in \Or(\Lambda_X,\Delta,H)$. 
    Since $W$ acts trivially on $P^\perp$, we have $g|_{P^\perp} = (\sigma g)|_{P^\perp}$.
    This shows $\widetilde\Gamma_W|_{P^\perp} \subset \widetilde\Gamma_T$.
\end{proof}

By \cite[Proposition A.5]{yu2020fourfolds}, we have:
\begin{prop}
    The natural map $\pi_T\colon \Gamma_T\bs \DD_T \to \Gamma_1\bs \DD_1$ is a normalization of its image.
\end{prop}

Define $\calH_T^*\coloneqq\calH_\infty\cap\DD_T$.
We will show that $\calH_T^*$ is a $\Gamma_T$-invariant hyperplane arrangement in $\DD_T$, then describe the Looijenga compactification of $\Gamma_T\bs(\DD_T-\calH_T^*)$.
\begin{prop}
\label{prop: calH_* subset calH_T}
    We have $\calH_T^*\subset\calH_T$.
    In particular, $\DD_T\not\subset\calH_\infty$.
\end{prop}
\begin{lem}
\label{lem: div(root in L)=1}
    Any root in $L$ has divisibility $1$ in $H^\perp$.
\end{lem}
\begin{proof}
    Assume that there exists a root $r\in L$ with divisibility $2$ in $H^\perp$.
    We will show that $u\coloneqq {1\over2}(H+r)\in P$.
    Since $u^2=0$, $u\cdot H=1$, this contradicts the fact that $(\Lambda_X,H,\Delta)$ is an abstract singular type.

    We have $H^\perp\cong A_1\oplus U^{\oplus 2}\oplus E_8^{\oplus 3}$ via $\phi$.
    Let $x\coloneqq\phi^{-1}(e_1-f_1)\in\Lambda_X$, then ${1\over2}(H+x)=\phi^{-1}(e_1)\in\Lambda_X$.
    The set of roots in $A_1\oplus U^{\oplus 2}\oplus E_8^{\oplus 3}$ forms two $\Or(H^\perp)$-orbits, which exactly corresponds to divisibility $1$ and $2$.
    Since the divisibility of $x$ is $2$, there exists $g\in\Or(H^\perp)$ such that $g x=r$.
    Since $A_{H^\perp}\cong \ZZ/2$, $g$ acts trivially on $A_{H^\perp}$.
    Thus $g$ extends to an element in $\Or(\Lambda_X,H)$, still denoted by $g$.
    Hence $u=g({1\over2}(H+x))$ lies in $\Lambda_X$, then in $P$.
\end{proof}
\begin{proof}[Proof of Proposition \ref{prop: calH_* subset calH_T}]
    For any $\omega\in\calH_T^*$, by the definition of $\calH_\infty$, there exists a root with divisibility $2$ in $\langle \omega,H \rangle^\perp$.
    By Lemma \ref{lem: div(root in L)=1}, the root does not lie in $L$.
    Hence $\calH_T^*\subset\calH_T$.
\end{proof}

By Lemma \ref{lem: Gamma_W|_P^perp = Gamma_T}, $\pi_T$ have a natural restriction $\pi_T\colon\Gamma_T\bs(\DD_T-\calH_T^*)\to\Gamma_1\bs(\DD_1-\calH_\infty)$.
By \cite[Theorem A.13]{yu2020fourfolds} and Lemma \ref{lem: Stab(x) = W}, we have 
\begin{prop}
\label{prop: Loo comp: type T to total}
    There is a natural extension of $\pi_T$ to Looijenga compactifications
    \begin{equation*}
        \pi_T\colon \overline{\Gamma_T\bs\DD_T}^{\calH_T^*}\to \overline{\Gamma_1\bs\DD_1}^{\calH_\infty},
    \end{equation*}
    which is the normalization of its image.
\end{prop}

\subsubsection{Identification between Compactifications}
\begin{prop}
\label{prop: comm diag, open to compact}
    The map $\pi_T\colon\Gamma_T\bs(\DD_T-\calH_T)\to\Gamma_1\bs(\DD_1-\calH_\infty)$ is injective.
    We have the following commutative diagram:
    \begin{equation*}
    \begin{tikzcd}
      \calM_T  \arrow[r,"\Prd_T","\cong"'] \arrow[d,hook,"j_T"] & \Gamma_T\bs(\DD_T-\calH_T) \arrow[d,hook,"\pi_T"] \\
      \overline{\calM} \arrow[r,"\Prd","\cong"'] & \overline{\Gamma_1\bs\DD_1}^{\calH_\infty},
    \end{tikzcd}
    \end{equation*}
    where the horizontal maps are both isomorphism and the vertical maps are both injective.
\end{prop}
\begin{proof}
    By Lemma \ref{lem: Gamma_P^perp = Gamma_W} and Lemma \ref{lem: Gamma_W|_P^perp = Gamma_T}, we have $\Gamma_T = \{ g|_{P^\perp} \,|\, g\in\Gamma_1, g(P^\perp)=P^\perp \}$.
    Namely, for any $\omega,\omega'\in\DD_T\subset\DD_1$, if there exists a $g\in\Gamma_1$ such that $g\cdot\omega=\omega'$, then $g|_{P^\perp}\in\Gamma_T$.
    Hence $\pi_T\colon\Gamma_T\bs(\DD_T-\calH_T)\to\Gamma_1\bs(\DD_1-\calH_\infty)$ is injective.
    
    The period map $\Prd\colon\calM^{\mathrm{ADE}}\xrightarrow[]{\cong}\Gamma_1\bs(\DD_1-\calH_\infty)$ is compatible with the period map of marked complex K3 surfaces.
    By Lemma \ref{lem: Prd_T coincides with global period map}, $\Prd_T\colon\calM_T\xrightarrow[]{\cong}\Gamma_T\bs(\DD_T-\calH_T)$ is also compatible with the period map of marked complex K3 surfaces.
    By Theorem \ref{thm: open}, Theorem \ref{thm: Shah--Looijenga}, together with $j_T\colon \calM_T\hookrightarrow\calM^{\mathrm{ADE}}\hookrightarrow\overline{\calM}$ and $\pi_T\colon \Gamma_T\bs(\DD_T-\calH_T)\hookrightarrow\Gamma_1\bs(\DD_1-\calH_\infty)\hookrightarrow\overline{\Gamma_1\bs\DD_1}^{\calH_\infty}$, we have the claimed commutative diagram.
\end{proof}

In the end of this section, we establish an identification between the GIT compactification $\widehat{\calM}_T$ of $\calM_T$ and the corresponding Looijenga comapactification.
For the construction and properties of $\widehat{\calM}_T$, see \S \ref{subsection: ADE sextic}.
\begin{thm}
\label{thm: comm diag, compact to compact}
    The period map $\Prd_T\colon\calM_T\xrightarrow[]{\cong}\Gamma_T\bs(\DD_T-\calH_T)$ extends to an isomorphism between the GIT compactification and the Looijenga compactification $\Prd_T\colon \widehat{\calM}_T\to\overline{\Gamma_T\bs\DD_T}^{\calH_T^*}$.
    We have the following commutative diagram:
    \begin{equation*}
    \begin{tikzcd}
      \widehat{\calM}_T \arrow[r,"\Prd_T","\cong"'] \arrow[d,"j_T"] & \overline{\Gamma_T\bs\DD_T}^{\calH_T^*} \arrow[d,"\pi_T"] \\
      \overline{\calM} \arrow[r,"\Prd","\cong"'] & \overline{\Gamma_1\bs\DD_1}^{\calH_\infty},
    \end{tikzcd}
    \end{equation*}
    where the two vertical maps are normalizatons of their images.
\end{thm}
\begin{proof}
    By Proposition \ref{prop: comm diag, open to compact}, $j_T(\calM_T)$ and $\pi_T(\Gamma_T\bs(\DD_T-\calH_T))$ are isomorphic via $\Prd$.
    Then their (Zariski) closures in $\overline{\calM}$ and $\overline{\Gamma_1\bs\DD_1}^{\calH_\infty}$ are isomorphic via $\Prd$.
    Thus the normalizations of the two closures, namely $\widehat{\calM}_T$ and $\overline{\Gamma_T\bs\DD_T}^{\calH_T^*}$, are isomorphic via a unique extension of $\Prd_T\colon\calM_T\xrightarrow[]{\cong}\Gamma_T\bs(\DD_T-\calH_T)$, which is still denoted by $\Prd_T$.
    
    By Proposition \ref{prop: M_T M^bar normalization} and Proposition \ref{prop: Loo comp: type T to total}, $j_T$ and $\pi_T$ extend to the compactifications $\widehat{\calM}_T$ and $\overline{\Gamma_T\bs\DD_T}^{\calH_T}$.
    By Proposition \ref{prop: comm diag, open to compact}, the compactified diagram is also commutative.
\end{proof}

\section{More Explicit Characterization of Lattices}
\label{section: explicit description of lattices}
In this section, we consider a plane curve of even degree.
For a singular type $T$, let $F$ be an element in $\calV_T$. 
We omit the index $F$ of all relevant symbols in this section for simplicity. 
Our main goal is to give an explicit description of the lattice $P=P_F$.

\subsection{A Topological Lemma}
We begin with a general lemma for branched double covers of smooth real $4$-manifolds.
Let $p\colon X\to S$ be a double cover with branched locus given by the disjoint union of smooth submanifolds $C_1,\cdots,C_l \subset S$ of real codimension two.
Denote by $\iota$ the deck transformation of $p$. 
Let $H^2(X,\ZZ)^\iota$ be the fixed sublattice of the induced action of $\iota$ on $H^2(X,\ZZ)$.
Assume that both $H_1(X,\ZZ)$ and $H_1(S,\ZZ)$ are torsion free.
In particular, any simply connected closed $4$-manifold satisfies the condition.

\begin{lem}
\label{lemma: double cover manifold}
    The double cover $p$ induces a finite-index extension $H^2(S,\ZZ)(2)\hookrightarrow H^2(X,\ZZ)^\iota$ such that
    \begin{equation*}
        \frac{H^2(X,\ZZ)^\iota}{H^2(S,\ZZ)(2)}\cong (\ZZ/2)^{l-1}.
    \end{equation*}
\end{lem}
\begin{proof}
This lemma is a slight generalization of \cite[Prop $3.4$]{yu2023moduli} and the proof is similar. We provide the proof here for completeness.

The universal coefficient theorem implies that the torsion part of $H^k(S,\ZZ)$ is isomorphic to the torsion part of $H_{k-1}(S,\ZZ)$. Thus $H^2(S,\ZZ)$ is torsion free. So does $H^2(X,\ZZ)$.
Let $M$ be the lattice generated by $H^2(S,\ZZ)(2)$ and $[\pi^{-1}C_1],\cdots,[\pi^{-1}C_l]$.
    By \cite[Lemma $3.3$]{yu2023moduli}, we have $M\subset H^2(X,\ZZ)^\iota$ with the same rank, denoted by $r$.
    Note that the rank of $H^2(S,\ZZ)$ also equals $r$.
    We have $A_{H^2(S,\ZZ)(2)}\cong (\ZZ/2)^{r}$ because $H^2(S,\ZZ)$ is unimodular.
    By the consecutive inclusions $H^2(S,\ZZ)(2)\subset M\subset H^2(X,\ZZ)^\iota\subset (H^2(S,\ZZ)(2))^{\vee}$, the quotients $H^2(X,\ZZ)^\iota/H^2(S,\ZZ)(2)$ and $M/H^2(S,\ZZ)(2)$ are isomorphic to certain powers of $\ZZ/2$.

    By Poincar\'e duality, $H_1(S,\ZZ)$ being torsion free implies that $H^3(S,\ZZ)$ is as well. 
    By the universal coefficient theorem, $H^2(S,\ZZ/2)$ is isomorphic to the direct sum of $H^2(S,\ZZ)\otimes\ZZ/2$ and the $2$-torsion part of $H^3(S,\ZZ)$. 
    The same results also hold for $X$.
    Tensoring the short exact sequence 
    \begin{equation*}
        0\to H^2(S,\ZZ)\to H^2(X,\ZZ)\to H^2(X,\ZZ)/H^2(S,\ZZ)(2)\to 0
    \end{equation*}
    with $\ZZ/2$, we obtain a long exact sequence 
    \begin{align*}
        0 & \to \Tor(H^2(X,\ZZ)/H^2(S,\ZZ)(2),\ZZ/2)\to H^2(S,\ZZ/2)\to H^2(X,\ZZ/2) \\
        & \to (H^2(X,\ZZ)/H^2(S,\ZZ)(2))\otimes \ZZ/2\to 0.
    \end{align*}
    The $\ZZ/2$-rank of the kernel of $H^2(S,\ZZ/2)\to H^2(X,\ZZ/2)$ equals to the $\ZZ/2$-rank of the $2$-torsion part (in this case also the whole torsion part) $(H^2(X,\ZZ)/H^2(S,\ZZ)(2))_{tor}$ of \\$H^2(X,\ZZ)/H^2(S,\ZZ)(2)$.
    By \cite[Theorem 1]{lee1995homology}, we have the exact sequence 
    \begin{equation*}
        0\to H^1(S,\bigcup_{i=1}^{l}C_l,\ZZ/2)\to H^2(S,\ZZ/2)\to H^2(X,\ZZ/2).
    \end{equation*}
    There is a long exact sequence for the relative cohomology groups
    \begin{equation*}
        0\to H^0(S,\ZZ/2)\to H^0(\bigcup_{i=1}^{l}C_i,\ZZ/2)\to H^1(S,\bigcup_{i=1}^l C_i,\ZZ/2)\to 0.
    \end{equation*}
    It follows that the $\ZZ/2$-rank of $H^1(S,\bigcup_{i=1}^l C_i,\ZZ/2)$ is $l-1$.
    From the above exact sequences we conclude that the $\ZZ/2$-rank of $(H^2(X,\ZZ)/H^2(S,\ZZ)(2))_{tor}$ is $l-1$.
    Since $H^2(X,\ZZ)^{\iota}$ is the primitive hull of $H^2(S,\ZZ)$ in $H^2(X,\ZZ)$, we have $H^2(X,\ZZ)^{\iota}/H^2(S,\ZZ)(2)\cong (\ZZ/2)^{l-1}$.
\end{proof}


\subsection{Involution of Root Lattice}
Let $Z=Z(F)$ be a plane curve of even degree. Suppose $Z$ has only simple singularities. 
Let $\widehat{X}$ denote the double cover of $\PP^2$ branched along $Z$, and $X$ denote the minimal resolution of $\widehat{X}$.
Similarly to the sextic case, there is a regular involution $\iota$ of $X$, with quotient $S=X/\iota$. Denote by $f\colon S\to \PP^2$ the minimal blowup such that the strict transform $\widetilde{Z}\subset S$ of $Z$ is smooth. The double cover $X\to S$ is branched along the disjoint union of $\widetilde{Z}$ and possibly certain exceptional curves. We have $H^2(S,\ZZ)\cong \left\langle f^* h\right\rangle\oplus T$, where $T$ arises from the blowup $S\to \PP^2$.

For a root lattice of type $R$, we denote by $w_0(R)$ the longest element of the Weyl group $W(R)$ with respect to a given base.
Recall that the root lattice $L=L_F$ is the direct sum of all $L(R)$, where $R$ runs through all singularity types (with multiplicities) appearing in $\widehat{X}$. 
The longest element $w_0(L)\in W(L)$ is equal to the product of all $w_0(R)$. We have $-w_0(L)\in \Or(L)$. For simplicity, the induced $\iota$-action on $H^2(X,\ZZ)$ (and its restrictions on subspaces) is still denoted by $\iota$. Denote by $L^\iota$ the sublattice fixed by $\iota$. 
\begin{prop}
We have $\iota=-w_0(L)$ as automorphisms of $L$.
\end{prop}
\begin{proof}
It suffices to check this for every irreducible root lattice $L$. 
Notice that $-w_0(L)$ is an involution preserving the base $\Delta$. The graph automorphisms induced by $w_0(L)$ are listed in \cite[PLATE I--X]{bourbaki2002lie4-6}.
\end{proof}

\subsection{Comparison of Lattices}
The double cover $p\colon X\to S$ induces an injective morphism $p^*\colon H^2(S,\ZZ)(2)\to H^2(X,\ZZ)$ of lattices, which maps $f^* h$ to $H$. By Propsition \ref{prop: finite index H^2(S) to H oplus L^iota}, it induces a finite-index extension $T(2)\hookrightarrow L^{\iota}$. 
\begin{lem}
\label{lem: T_F and L_F}
We have $L^{\iota}/T(2)\cong (\ZZ/2)^m$, where $m$ is the number of exceptional curves in $S$ lying in the branch locus of $X\to S$.
\end{lem}
\begin{proof}
It suffices to check this for every irreducible root lattice $L$. Suppose $D$ is a branch exceptional curve in $S$, then the Chern class of $\pi^{-1}(D)$ lies in $L^\iota$, while twice of it lies in $T(2)$. This implies the required isomorphism.
\end{proof}

Table \ref{table: m_R} provides the values of $m$ for root lattices of ADE types.

\begin{table}[htbp]  
  \centering  
  \caption{}  
  \begin{tabular}{cccccc}
  \toprule
  Root Type $R$ & $A_n$ & $D_n$ & $E_6$ & $E_7$ & $E_8$ \\  
  \midrule
  $m_R$ & $0$ & $[\frac{n-2}{2}]$ & $1$ & $3$ & $4$ \\  
  \bottomrule
  \end{tabular}
  \label{table: m_R}  
\end{table}

\begin{cor}
\label{cor: if only A_1,D_2n,E_7,E_8}
    If the plane sextic curve $Z$ has only singularities of type $A_1,D_{2n}(n\ge 2),E_7,E_8$, then $L^\iota=L$. If only type $A_1$ appears, then $T(2)\cong L$ as lattices.
\end{cor}


For a root lattice $R$ of ADE type, we give a description of $R^{w_0(R)}$, which is the anti-invariant part of the $-w_0(R)$-action.
Denote by $r_R^-$ the rank of $R^{w_0(R)}$.
\begin{lem}
\label{lem: anti-inv}
    The discriminant group of $R^{w_0(R)}$ for each ADE-type $R$ is listed in Table \ref{tab: r_R^-}.
\begin{table}[htbp]  
  \centering  
  \caption{}  
  \begin{tabular}{ccccc}
  \toprule
  Root Type $R$ & $A_{2n-1}$ & $A_{2n}$ & $D_{n}$ (odd) & $E_6$  \\  
  \midrule
  $r_R^-$ & $n-1$ & $n$ & $1$ & $2$ \\ 
  $A_{R^{w_0(R)}}$ & $(\ZZ/2)^{n-2}\times \ZZ/2n$ & $(\ZZ/2)^n\times \ZZ/(2n+1)$ & $\ZZ/4$ & $\ZZ/3\times(\ZZ/2)^2$ \\  
  \bottomrule
  \end{tabular}
  \label{tab: r_R^-}  
\end{table}
For the type $A_1$, $D_n$ ($n$ even), $E_7$ and $E_8$, $R^{w_0(R)}$ is trivial.
\end{lem}

Denote by $D_1,\cdots,D_{l'}$ the irreducible components of $Z\subset \PP^2$.
Let $\widetilde{D}_1,\cdots,\widetilde{D}_{l'}$ be the connected components of $\widetilde{Z}\subset S$.
Denote by $E_1,\cdots,E_{m}$ the exceptional curves of $S\to \PP^2$ lying in the branched locus of $\pi\colon X\to S$.
The double cover $\pi$ is branched along the disjoint union of $\widetilde{D}_1,\cdots,\widetilde{D}_{l'}$ and $E_1,\cdots,E_{m}$.
Let $l\coloneqq l'+m$.

Recall $P=P_F$ is the primitive hull of $\left\langle H\right\rangle \oplus L$ in $H^2(X,\ZZ)$.
Let $M$ be the sublattice of $H^2(X,\ZZ)$ generated by $H$, $L$ and Chern classes of curves $\pi^{-1}(\widetilde{D}_1),\cdots, \pi^{-1}(\widetilde{D}_{l'}) \subset X$.
Denote by $\beta_i$ the Chern class of $\pi^{-1}(\widetilde{D}_i)\subset X$.


\begin{lem}
\label{lem: saturation of inv part of P}
    We have
    \begin{equation*}
        \frac{P^{\iota}}{(\left\langle H \right\rangle\oplus L)^{\iota}} \cong (\ZZ/2)^{l'-1}.
    \end{equation*}
\end{lem}
\begin{proof}
By definition $\beta_i\in H^2(X, \ZZ)^\iota$ for $1\le i\le l'$. 
Denote by $N$ the lattice $H^2(S,\ZZ)(2)=(\left\langle f^*h \right\rangle \oplus T)(2)$, which is a sublattice of $H^2(X, \ZZ)$.
By Lemma \ref{lemma: double cover manifold}, $H^2(X,\ZZ)^{\iota}$ is the primitive hull of $N$ in $H^2(X, \ZZ)$.
Since $N\subset \left\langle H\right\rangle \oplus L$, we have $H^2(X,\ZZ)^{\iota}\subset P$. It follows $H^2(X,\ZZ)^{\iota}=P^\iota$.
Consequently, we have $M\subset P$.
Then there are successive finite-index extensions
\begin{equation}
\label{eqn: iota-invariant successive saturations}
    N \hookrightarrow (\left\langle H \right\rangle\oplus L)^{\iota} \hookrightarrow M^{\iota} \hookrightarrow P^{\iota}.
\end{equation}
By Lemma \ref{lemma: double cover manifold}, we have $P^{\iota} / N\cong (\ZZ/2)^{l-1}$.
By Lemma \ref{lem: T_F and L_F}, we have $(\left\langle H \right\rangle\oplus L)^{\iota}/N \cong (\ZZ/2)^{m}$. 
Thus the lemma follows.
\end{proof}

\begin{lem}
\label{lem: Z/2-independence}
    Any $l'-1$ elements among $\beta_1, \cdots, \beta_{l'}$ are $\ZZ/2$-independent in $\frac{M^{\iota}}{(\left\langle H \right\rangle\oplus L)^{\iota}}$.
\end{lem}
\begin{proof}
The nodal case is proved in \cite[Proposition 3.4]{yu2023moduli}. 
Without loss of generality, we consider the first $l'-1$ elements $\beta_1,\cdots,\beta_{l'-1}$.
Consider the set $A=(\mathop\bigcup\limits_{i=1}^{l'-1}D_i)\cap D_{l'}$. The singularities in $A$ must be of type $A_{2n-1}, D_n, E_7$. It follows from the condition that each singularity in $A$ has at least two local components.

It suffices to show that any nontrivial linear combination of $\beta_1, \cdots, \beta_{l'-1}$ with coefficients $0$ or $1$ does not belong to $\langle H\rangle \oplus L$. 
It is enough to consider $\Sigma=\beta_1+\cdots+\beta_k$, where $k$ is an integer satisfying $1\le k\le l'-1$. Suppose there is a singularity of type $A_{2n-1}$ on $(\mathop\bigcup\limits_{i=1}^{k} D_i)\cap D_{l'}$. 
The singularity lies on exactly two irreducible components of the divisor, which we denote by $D_i$ and $D_{l'}$.
The graph of exceptional curves and strict transforms of $D_i$, $D_{l'}$ is drawn in Figure \ref{figure: A_2n-1} (also see \cite[Chapter III, \S 7, Table 1]{barth2004compact}). 
\begin{figure}[htp]
  \centering
  \includegraphics[width=13cm]{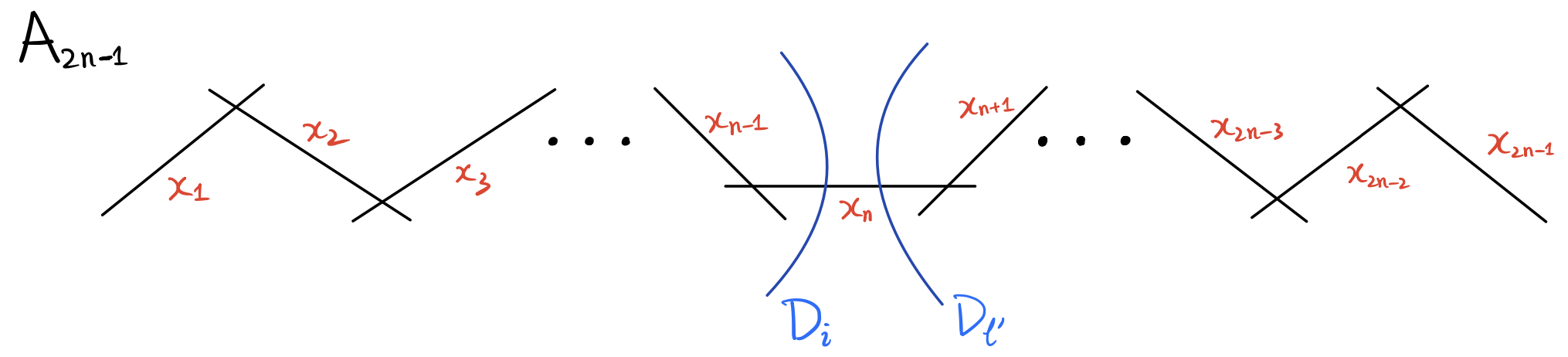}
  \caption{Resolution graph of $A_{2n-1}$-type}
  \label{figure: A_2n-1}
 \end{figure}
Both $\beta_i$ and $\beta_{l'}$ intersect only the middle curve.
Hence the projection of $\beta_i$ (also $\beta_{l'}$) to the root lattice of type $A_{2n-1}$ is the fundamental weight $\lambda_n$. The weight $\lambda_n$ can be expressed as
\begin{equation*}
    \lambda_n = -{1\over 2}x_1 - {2\over 2}x_2 - {3\over 2}x_3 - \cdots - {n\over 2}x_n - {n-1\over 2}x_{n+1} - \cdots - {1\over 2}x_{2n-1},
\end{equation*}
where the coefficients of both $x_i$ and $x_{2n-i}$ $(1\le i\le n)$ are $-{i\over 2}$.
(A general calculation of fundamental weights can be found in \cite[PLATE I--X]{bourbaki2002lie4-6}.)
This is exactly the projection of $\Sigma$ to the associated $A_{2n-1}$-root lattice. Thus $\Sigma\notin \langle H\rangle\oplus L$.

The cases of types $D_n$ and $E_7$ are similar, but a bit more involved.
We call the exceptional curve corresponding to $m$-node the $m$-line.
See Figure \ref{figure: D_2n-1}, \ref{figure: D_2n}, \ref{figure: E_7} for the numbering of each node.
\begin{figure}[htp]
  \centering
  \includegraphics[width=8cm]{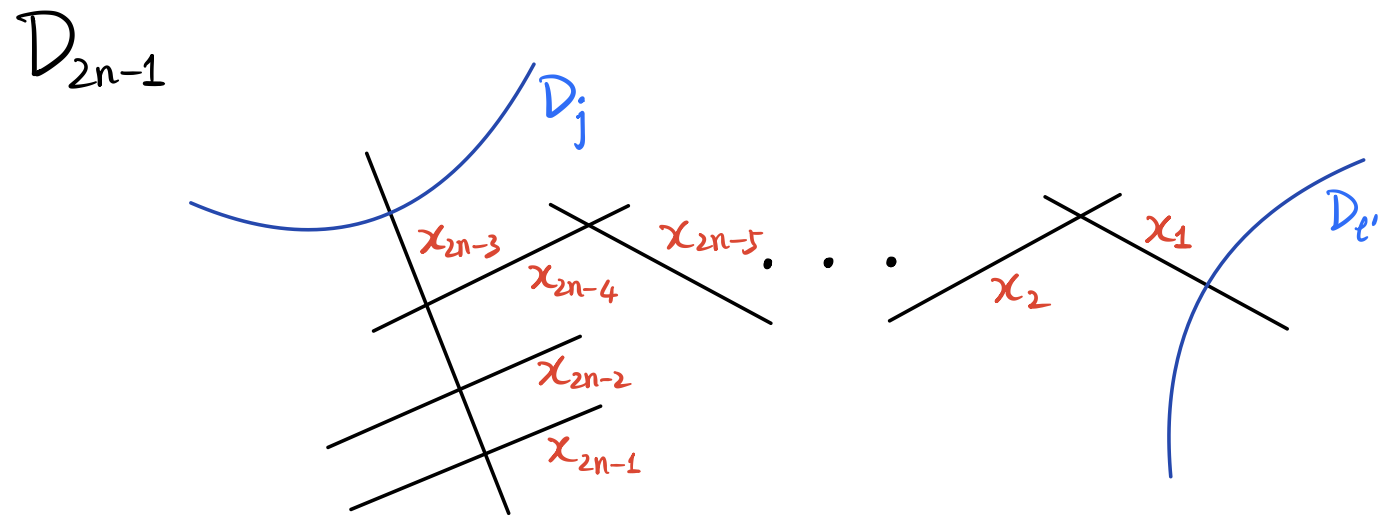}
  \caption{Resolution graph of $D_{2n-1}$-type}
  \label{figure: D_2n-1}
\end{figure}
\begin{figure}[htp]
  \centering
  \includegraphics[width=9cm]{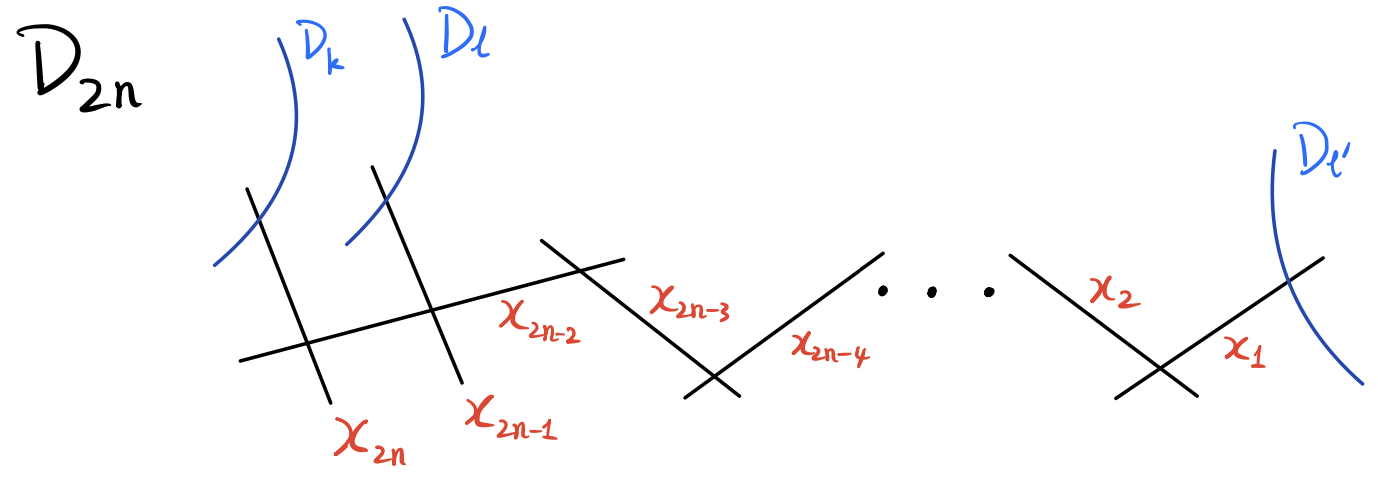}
  \caption{Resolution graph of $D_{2n}$-type}
  \label{figure: D_2n}
\end{figure} 
\begin{figure}[htp]
  \centering
  \includegraphics[width=7cm]{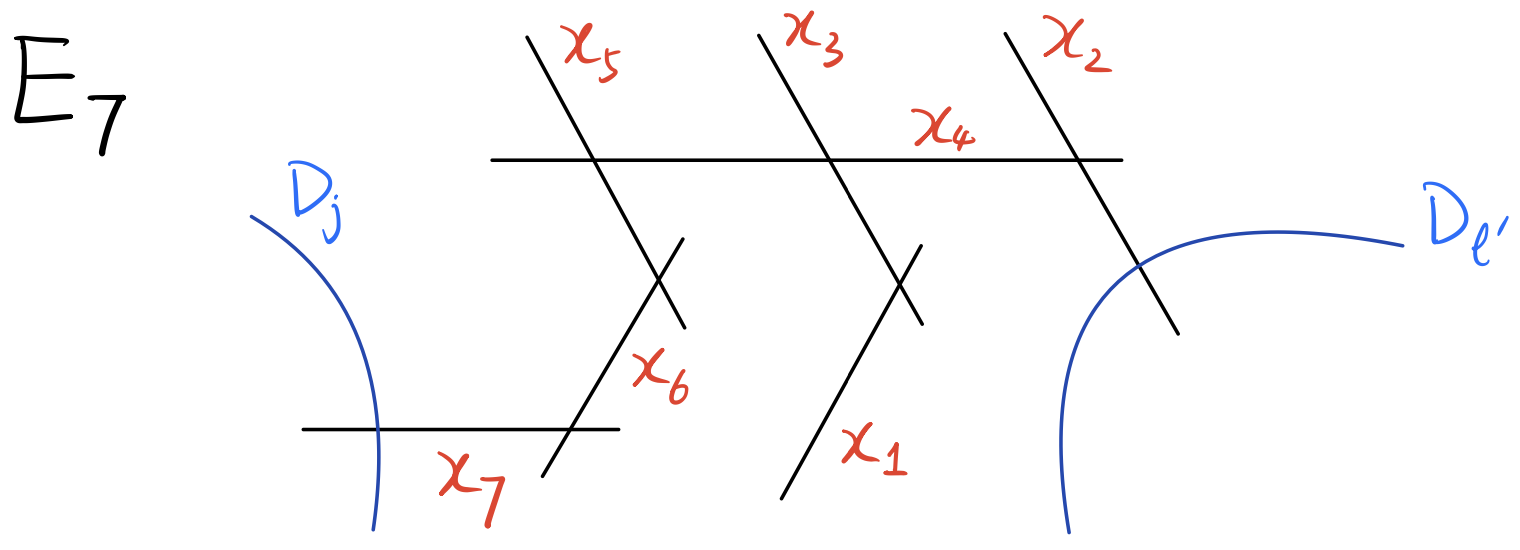}
  \caption{Resolution graph of $E_7$-type}
  \label{figure: E_7}
\end{figure}
 
A singularity of type $D_{2n-1}$ or $E_7$ lies on exactly two irreducible components, which we denote by $D_j$ and $D_{l'}$.
In $D_{2n-1}$ case, the strict transform of the locally cuspidal component intersects only the $(2n-3)$-line. 
The strict transform of the locally smooth component intersects only the $1$-line. 
The projection of $\beta_i$ to the root lattice of type $D_{2n-1}$ is the fundamental weight $\lambda_1$ or $\lambda_{2n-3}$. The weights $\lambda_1$ and $\lambda_{2n-3}$ have the following expressions
\begin{align*}
    \lambda_{2n-3} &= -x_1-2x_2-3x_3-\cdots-(2n-3)x_{2n-3}-{2n-3\over 2}(x_{2n-2}+x_{2n-1}), \\
    \lambda_{1} &= -(x_1+x_2+\cdots+x_{2n-3})-{1\over 2}(x_{2n-2}+x_{2n-1}).
\end{align*}
In $E_7$ case, the strict transform of the locally cuspidal component intersects only the $2$-line. 
The strict transform of the locally smooth component intersects only the $7$-line. 
The weights $\lambda_2$ and $\lambda_7$ have the following expressions 
\begin{align*}
    \lambda_{2} &= -{1\over 2}(4x_1+7x_2+8x_3+12x_4+9x_5+8x_6+3x_7), \\
    \lambda_{7} &= -{1\over 2}(2x_1+3x_2+4x_3+6x_4+5x_5+4x_6+3x_7).
\end{align*}
These are exactly the projections of $\Sigma$ to the associated $D_{2n-1}$ and $E_7$. Thus $\Sigma\notin \langle H\rangle\oplus L$.

A singularity of type $D_{2n}$ lies on exactly three irreducible components, which we denote by $D_k$, $D_l$ and $D_{l'}$.
The strict transform of two components intersect only with the $2n$-line and the $(2n-1)$-line respectively. 
The strict transform of the rest component intersects only the $1$-line.
The corresponding fundamental weights can be expressed as
\begin{align*}
    \lambda_1 &= -(x_1+x_2+\cdots+x_{2n-2})-{1\over 2}(x_{2n-1}+x_{2n}), \\
    \lambda_{2n-1} &= -{1\over 2}x_1 - {2\over 2}x_2 - {3\over 2}x_3 - \cdots - {2n-2\over 2}x_{2n-2} - {n\over 2}x_{2n-1} - {n-1\over 2}x_{2n}, \\
    \lambda_{2n} &= -{1\over 2}x_1 - {2\over 2}x_2 - {3\over 2}x_3 - \cdots - {2n-2\over 2}x_{2n-2} - {n-1\over 2}x_{2n-1} - {n\over 2}x_{2n}. \\
\end{align*}
The strict transform of $D_{l'}$ may give one or two components in the resolution graph of the singularity.
Every linear combination of the rest components with coefficients $0$ or $1$ has non-integer coefficients.
These are exactly the projections of $\Sigma$ to the associated $D_{2n}$. Thus $\Sigma\notin \langle H\rangle\oplus L$.



    By the above calculation, any nontrivial $\ZZ/2$-combination of $\beta_1,\cdots\beta_{l'-1}$ does not lie in $\langle H\rangle\oplus L$.
    Hence $\beta_1,\cdots\beta_{l'-1}$ is $\ZZ/2$-independent in $\frac{M^{\iota}}{(\left\langle H \right\rangle\oplus L)^{\iota}}$.
\end{proof}
\begin{rmk}
    The above calculation in the case of sextic curves can also be found in \cite[\S 3, Tabel 1]{yang1996sextic}.
\end{rmk}

\begin{prop}
\label{prop: saturation of M_F}  
We have 
    \begin{equation*}
        \frac{M}{\left\langle H\right\rangle \oplus L}\cong (\ZZ/2)^{l'-1}.
    \end{equation*}
    This finite group is generated by equivalence classes of $\beta_1,\cdots, \beta_{l'-1}$.
    Moreover, $M^{\iota} = P^{\iota}$.
\end{prop}
\begin{proof}
Since $\beta_i \in M^{\iota}$ for $1\le i\le l'$, the finite-index extensions in (\ref{eqn: iota-invariant successive saturations}) induces injective maps of the following finite abelian groups
\begin{equation*}
    \frac{M^{\iota}}{(\left\langle H \right\rangle\oplus L)^{\iota}} \hookrightarrow \frac{P^{\iota}}{(\left\langle H \right\rangle\oplus L)^{\iota}} \hookrightarrow \frac{P}{\left\langle H \right\rangle\oplus L}.
\end{equation*}
By Lemma \ref{lem: saturation of inv part of P}, we have 
\begin{equation*}
    \frac{M^{\iota}}{(\left\langle H \right\rangle\oplus L)^{\iota}} = \frac{P^{\iota}}{(\left\langle H \right\rangle\oplus L)^{\iota}} \cong (\ZZ/2)^{l'-1}.
\end{equation*}
By Lemma \ref{lem: Z/2-independence}, there is only one nontrivial relation among $\beta_1, \cdots, \beta_{l'}$. 
This isomorphism implies that $M^{\iota} = P^{\iota}$.
Moreover, the inclusion 
\begin{equation*}
    \frac{M^{\iota}}{(\left\langle H \right\rangle\oplus L)^{\iota}} \hookrightarrow \frac{M}{\left\langle H \right\rangle\oplus L}
\end{equation*}
is surjective (hence an isomorphism) because $\beta_1, \cdots, \beta_{l'}$ lies in $M^\iota$. We have done.
\end{proof}

\begin{lem}
\label{lem: saturation induced by involution, inequality and ADE case}
    Let $L$ be an even lattice and $\iota\in \Or(L)$ be an involution.
    Then $\frac{L}{L^\iota \oplus L^{-\iota}}$ is $2$-elementary, and its cardinality satisfies
    \begin{equation*}
        \bigg|\frac{L}{L^\iota \oplus L^{-\iota}}\bigg| \le 2^{min\{\rank (L^\iota), \rank (L^{-\iota})\}}.
    \end{equation*}
    In particular, for every root lattice $L$ of ADE type, the equality is achieved, i.e. 
    \begin{equation*}
        \frac{L}{L^{\iota} \oplus L^{-\iota}} \cong (\ZZ/2)^{\rank (L^{-\iota})}.
    \end{equation*}
\end{lem}
\begin{proof}
    For any $v\in L$, we have $2v = (v+\iota(v))+(v-\iota(v))$, hence $2v\in L^{\iota} \oplus L^{-\iota}$.
    Thus $\frac{L}{L^\iota \oplus L^{-\iota}}$ is $2$-elementary.
    Denote by $r\coloneqq min\{\rank (L^\iota), \rank (L^{-\iota})\}$.
    For any $n>r$, we assume that there exist $\frac{v_1+w_1}{2},\cdots,\frac{v_n+w_n}{2}$ in $L\bs(L^{\iota} \oplus L^{-\iota})$.
    Without loss of generality, we assume that $\rank (L^\iota) \ge \rank (L^{-\iota})$.
    Since $n > r=\rank (L^{-\iota})$, there exist $a_1,\cdots,a_n$ ($\exists a_i\neq 0$) such that $\sum\limits_{i=1}^n a_i w_i = 0$.
    Then $\sum\limits_{i=1}^n a_i \frac{v_i+w_i}{2} = \sum\limits_{i=1}^n a_i \frac{v_i}{2}$ lies in $L^\iota$.
    Hence $\sum\limits_{i=1}^n a_i \frac{v_i+w_i}{2}$ vanishes in $\frac{L}{L^{\iota} \oplus L^{-\iota}}$, namely, there exists a nontrivial relation among $\frac{v_1+w_1}{2},\cdots,\frac{v_n+w_n}{2}$.
    Thus $\dim_{\ZZ/2}(\frac{L}{L^\iota \oplus L^{-\iota}}) \le \rank (L^{-\iota})$ (view $\frac{L}{L^\iota \oplus L^{-\iota}}$ as a $\ZZ/2$-vector space).

    For root lattices of ADE types, the equality arises from a case-by-case check.
    Notice that 
    \begin{equation*}
        \bigg|\frac{L}{L^\iota \oplus L^{-\iota}}\bigg|^2 = \frac{|A_{L^\iota} \oplus A_{L^{-\iota}}|}{|A_L|}.
    \end{equation*}
    For each ADE type, $A_{L^{-\iota}}$ is listed in Table \ref{tab: r_R^-}.
    We list $A_{L^\iota}$ in Table \ref{tab: r_R^+}.
    \begin{table}[htbp]  
    \centering  
    \caption{}  
    \begin{tabular}{ccccc}
    \toprule
    Root Type $R$ & $A_{2n-1}$ & $A_{2n}$ & $D_{n}$ (odd) & $E_6$  \\  
    \midrule
    $r_R^+$ & $n$ & $n$ & $n-1$ & $4$ \\ 
    $A_{R^{-w_0(R)}}$ & $(\ZZ/2)^{n}$ & $(\ZZ/2)^n$ & $(\ZZ/2)^2$ & $(\ZZ/2)^2$ \\  
    \bottomrule
    \end{tabular}
    \label{tab: r_R^+}  
    \end{table}
\end{proof}

\begin{prop}
\label{prop: M_F = Z_2-saturation}
    The lattice $M$ is the biggest finite-index extension of $P^\iota \oplus L^{-\iota}$ in $H^2(X,\ZZ)$ such that $M\over{P^{\iota}\oplus L^{-\iota}}$ is $2$-elementary. Moreover, $\frac{M}{P^{\iota}\oplus L^{-\iota}}\cong (\ZZ/2)^{r^-}$.
\end{prop}
\begin{proof}
    Since $P^{\iota}$ is a finite-index extension of $\left\langle H\right\rangle\oplus L^{\iota}$, $L^{-\iota}$ is orthogonal to $P^{\iota}$ in $P$.
    Consider the finite-index extension
    \begin{equation*}
        P^{\iota}\oplus L^{-\iota}\hookrightarrow P.
    \end{equation*}
    By Proposition \ref{prop: saturation of M_F}, we have $M^{\iota} = P^{\iota}$.
    Thus there is a natural finite-index extension 
    \begin{equation*}
        P^{\iota}\oplus L^{-\iota} \hookrightarrow M.
    \end{equation*}
    We have the following successive finite-index extensions
    \begin{equation*}
        \left\langle H \right\rangle \oplus L^{\iota} \oplus L^{-\iota} \hookrightarrow \left\langle H \right\rangle \oplus L \hookrightarrow M.
    \end{equation*}
    By Lemma \ref{lem: saturation induced by involution, inequality and ADE case}, we have $\frac{L}{L^{\iota} \oplus L^{-\iota}} \cong (\ZZ/2)^{r^-}$.
    By Proposition \ref{prop: saturation of M_F}, we have $\frac{M}{\left\langle H\right\rangle \oplus L}\cong (\ZZ/2)^{l'-1}$.
    Moreover, we have another successive finite-index extensions 
    \begin{equation*}
        \left\langle H \right\rangle \oplus L^{\iota} \oplus L^{-\iota} \hookrightarrow P^{\iota}\oplus L^{-\iota} \hookrightarrow M.
    \end{equation*}
    By Lemma \ref{lem: saturation of inv part of P} and Proposition \ref{prop: saturation of M_F}, we have $\frac{P^{\iota}}{\left\langle H \right\rangle \oplus L^{\iota}} = \frac{M}{\left\langle H \right\rangle \oplus L} \cong (\ZZ/2)^{l'-1}$.
    It is direct to verify that $(\left\langle H \right\rangle \oplus L)\cap (P^{\iota}\oplus L^{-\iota}) = \left\langle H \right\rangle \oplus L^\iota \oplus L^{-\iota}$.
    Thus $\frac{M}{\left\langle H \right\rangle \oplus L^\iota \oplus L^{-\iota}} \cong (\ZZ/2)^{l'+r^--1}$.
    Therefore we conclude
    \begin{equation*}
    \label{eqn: saturation of M over P^iota oplus L^-iota}
        \frac{M}{P^{\iota}\oplus L^{-\iota}}\cong (\ZZ/2)^{r^-},
    \end{equation*}
    where $r^-$ is the sum of all $r_R^-$ with $R$ that appear in $L$ (with multiplicities).

    Since $P^{\iota}$ is primitive in $P$, by lattice theory, the quotient $\frac{P}{P^{\iota}\oplus L^{-\iota}}$ is isomorphic to certain subgroup of $A_{L^{-\iota}}$. 
    By Lemma \ref{lem: anti-inv}, the biggest subgroup of each $A_{R^{w_0(R)}}$ of the form $(\ZZ/2)^k$ is isomorphic to $(\ZZ/2)^{r_R^-}$.
    Notice that $(\ZZ/2)^{r^-}$ can be naturally decomposed into the product $\prod\limits_{R}(\ZZ/2)^{l_T(R)r_R^-}$.
    Therefore $M$ is the biggest finite-index extension of $P^{\iota}\oplus L^{-\iota}$ such that $M\over{P^{\iota}\oplus L^{-\iota}}$ is $2$-elementary.
\end{proof}

\begin{rmk}
   Since $H$ is primitive in $H^2(X,\ZZ)$, $\frac{P}{\left\langle H \right\rangle \oplus L}$ is a subgroup of $A_L$. 
\end{rmk}

\begin{rmk}
    There exist examples such that $M$ is a proper sublattice of $P$, for example, see \S \ref{subsec: Zariski pair} and Propostion \ref{proposition: picard zariski pair}.
    By \cite[Proposition 3.4]{yu2023moduli}, $M=P$ holds for all nodal singular types.
    We will give more examples with $M=P$ in \S \ref{sec: examples and applications}.
\end{rmk}






\section{Orbifold Structures}
\label{sec: orbifold structures}
Given a singular type $T$, both the moduli space $\calM_T$ and the associated arithmetic quotient $\Gamma_T\bs(\DD_T-\calH_T)$ admit natural orbifold structures induced by their nontrivial automorphism groups. It is natural to study the relationship of these two orbifold structures. See \cite{kudla2012occult} and \cite{zheng2021orbifold} for some well-known cases. In this section, we compare the orbifold structures of the two sides of the period map $\Prd_T\colon \calM_T\to \Gamma_T\bs(\DD_T-\calH_T)$.
We follow the setup in \S \ref{section: period map}.
For a general introduction to orbifolds, one can see for instance \cite[Chapter 13]{thurston2022geometry}.

\subsection{A Lattice-theoretic Criterion}
For $Z(F)\in \PP\calV_T$, let $(X_F,H_F,\iota_F)$ be defined as in \S \ref{subsection: ADE K3}, and $\Delta_F$ be defined as in \S \ref{subsection: define period map}.

    

\begin{prop}
\label{prop: describe Stab_Z(F)}
    We have a natural isomorphism
    \begin{align*}
        \Stab_{Z(F)} &\cong \Aut(X_F,H_F)/\{\Id,\iota_F\} \\
        &\cong \Or(\Lambda_F,H_F,\Delta_F,H^{2,0}(X_F))/\{\Id,\iota_F^*\},
    \end{align*}
    where $\Stab_{Z(F)}$ is the stabilizer of $Z(F)$ in $\PGL(3)$. 
\end{prop}

\begin{lem}
\label{lem: describe Aut(X_F,H_F,Delta_F)}
    We have
    \begin{align*}
        \Aut(X_F,H_F,\Delta_F) &= \Aut(X_F,H_F) \\
        &\cong \Or(\Lambda_F,H_F,\Delta_F,H^{2,0}(X_F)).
    \end{align*}
\end{lem}
\begin{proof}
    Let $f\in \Aut(X_F)$ satisfy $f^*H_F=H_F$. Then $f^*L_F=L_F$, as $L_F$ is the root lattice of $\langle H_F,H^{2,0}(X_F)\rangle^\perp_{\Lambda_F}$.
    Since $f^*$ maps effective classes to effective classes, it follows that $f^*\Delta_F=\Delta_F$.
    Consequently, we conclude $\Aut(X_F,H_F,\Delta_F) = \Aut(X_F,H_F)$.

    Moreover, by the proof of Theorem \ref{theorem: global torelli}, we can construct an ample class $\alpha_F$ of $X_F$ lying in $\langle H_F\rangle\oplus L_F$.
    Since $\Aut(X_F,\alpha_F)\cong\Or(\Lambda_F,\alpha_F,H^{2,0}(X_F))$ (given by $f\mapsto f^*$) and $\Aut(X_F,H_F,\Delta_F) \subset \Aut(X_F,\alpha_F)$, it follows that $f\mapsto f^*$ gives $\Aut(X_F,H_F,\Delta_F) \cong \\ \Or(\Lambda_F,H_F,\Delta_F,H^{2,0}(X_F))$.
\end{proof}

\begin{lem}
\label{lem: Aut(Lambda_F,H_F,Delta_F) commutes with iota_F}
Each element $f\in \Or(\Lambda_{F}, \Delta_{F},H_{F})$ commutes with the involution $\iota_{F}$.
\end{lem}
\begin{proof}
Consider a singularity of type $R$ on $Z(F)$. The lattices $L(R)$ and $f(L(R))$, both of type $R$, satisfy $f(L(R)\cap \Delta_{F}) = f(L(R))\cap \Delta_{F}$. 
By Proposition \ref{prop: iota-action on resolution graph/base} and Remark \ref{rmk: iota-action on L(R)}, we can identify the generators of $L(R)^{\iota_{F}}$ and $f(L(R))^{\iota_{F}}$. 
This implies $f(L(R)^{\iota_{F}}) = f(L(R))^{\iota_{F}}$. 
Since $f(H_{F})=H_{F}$, we conclude $f(\Lambda_{F}^{\iota_{F}})=\Lambda_{F}^{\iota_{F}}$. 
Consequently, $f$ commutes with $\iota_{F}$.
\end{proof}

\begin{lem}
\label{lem: Aut(X_F,H_F) to Stab_Z(F)}
    Every automorphism $f\in \Aut(X_F,H_F)$ naturally induces an automorphism of $Z(F)\subset\PP^2$.
    The only automorphisms of $X_F$ inducing the identity map on $Z(F)$ are the involution $\iota_F$ and the identity map.
\end{lem}
\begin{proof}
    By Lemma \ref{lem: describe Aut(X_F,H_F,Delta_F)}, every $f\in \Aut(X_F,H_F)$ induces a $f^*\in\Or(\Lambda_F)$ that preserves $H_F$ and $\Delta_F$. 
    Lemma \ref{lem: Aut(Lambda_F,H_F,Delta_F) commutes with iota_F} ensures that $f^*$ commutes with $\iota_F^*$ on $\Lambda_F$. 
    By the injectivity of $\Aut(X_F)\hookrightarrow\Or(\Lambda_F)$, it follows that $f$ commutes with $\iota_F$.
    Hence $f$ induces an automorphism of $S_F$ that preserves the strict transform of $Z(F)$, which lies in the branched locus of $p_F\colon X_F\to S_F$.
    Moreover, since $f^*H_F=H_F$, $f$ descends to an automorphism $\overline{f}$ of $\PP^2$ that preserves $Z(F)$.

    Let $f\in\Aut(X_F)$ be an automorphism such that the induced map $\overline{f}$ restricts to the identity map on $Z(F)$.
    Then $\overline{f}$ is the identity map on $\PP^2$.
    The identity map lifts trivially to $S_F$, and thus $f$ must coincide with either the identity map on $X_F$ or the involution $\iota_F$.
\end{proof}
\begin{proof}[Proof of Proposition \ref{prop: describe Stab_Z(F)}]
    There exists a natural map 
    \begin{equation*}
        \Aut(X_F,H_F) \to \Stab_{Z(F)}, \quad f\mapsto\overline{f}.
    \end{equation*} 
    By Lemma \ref{lem: Aut(X_F,H_F) to Stab_Z(F)}, the kernel of this map is $\{\Id,\iota_F\}$. Moreover, every automorphism of $\PP^2$ preserving $Z(F)$ lifts to an automorphism of the double cover branched along $Z(F)$, and subsequently to an automorphism of $X_F$. Thus we obtain $\Stab_{Z(F)} \cong \Aut(X_F,H_F)/\{\Id,\iota_F\}$.

    By Lemma \ref{lem: describe Aut(X_F,H_F,Delta_F)}, there is a natural isomorphism
    \begin{equation*}
        \Aut(X_F,H_F)/\{\Id,\iota_F\} \cong \Or(\Lambda_F,H_F,\Delta_F,H^{2,0}(X_F))/\{\Id,\iota_F^*\}.
    \end{equation*}
    Combining these results, the proposition follows.
\end{proof}

Let $Z(F_0)\in\PP\calV_T$ be a base point.
By Corollary \ref{cor: Lambda_F^iota = P_F^iota}, the involution $\iota_{F_0}$ induces an action on $\Lambda_{F_0}$ preserving $H_{F_0}$ and $\Delta_{F_0}$, and acts as $-\Id$ on $Q_{F_0}$.
We denote by $P\Gamma_T$ the quotient group $\Gamma_T/\{\pm \Id\}$. The natural map $\Gamma_T \to \Aut(\DD_T)$ descends to $P\Gamma_T \hookrightarrow \Aut(\DD_T)$.  
We consider the orbifold $P\Gamma_T\bs(\DD_T-\calH_T)$ instead of $\Gamma_T\bs(\DD_T-\calH_T)$.
For $\omega\in \DD_T-\calH_T$, the stabilizer of $\omega$ in $P\Gamma_T$ is given by
\begin{equation*}
    \Stab_\omega = \{ g|_{Q_{F_0}} \,\big|\, g\in \Or(\Lambda_{F_0},H_{F_0},\Delta_{F_0}),\, g\cdot\omega=\omega \}/\{\pm \Id\},
\end{equation*}
where $\omega$ represents a positive complex line in $Q_{F_0}\otimes\CC$.
There exists a $Z(F)\in \PP\calV_T$ satisfying $\Prd_T(Z(F))=[\omega]$.
By definition of $\Prd_T$, there exists a path $\gamma_F\subset\PP\calV_T$ connecting $Z(F)$ and $Z(F_0)$ such that $\gamma_F^*H^{2,0}(X_F)=\omega$.
Let $\varphi$ denote $\gamma_F^*\colon H^2(X_F,\ZZ)\to H^2(X_{F_0},\ZZ)$.
The homomorphism $\varphi$ satisfies $\varphi(H_F)=H_{F_0}$, $\varphi(\Delta_F)=\Delta_{F_0}$.
Then we can define a group homomorphism
\begin{equation*}
    \Aut(X_F,H_F) \to \{ g|_{Q_{F_0}} \,\big|\, g\in \Or(\Lambda_{F_0},H_{F_0},\Delta_{F_0},\omega) \},\quad f \mapsto (\varphi\circ f^*\circ\varphi^{-1})|_{Q_{F_0}},
\end{equation*}
which is surjective.
It induces a surjective homomorphism $\Stab_{Z(F)}\to\Stab_\omega$.
\begin{prop}
\label{prop: orbifold criterion}
    The following statements are equivalent to each other.
    \begin{enumerate}[(i)]
        \item The period map $\Prd_T\colon \calM_T\to P\Gamma_T\bs(\DD_T-\calH_T)$ preserves the orbifold structures of the GIT quotient and the arithmetic quotient.
        \item 
        The natural restriction map $\Or(\Lambda_{F_0},H_{F_0},\Delta_{F_0})\to \Gamma_T$ is an isomorphism.
        \item The group $\Or(P_{F_0},H_{F_0},\Delta_{F_0})$ acts faithfully on $A_{P_{F_0}}$.
    \end{enumerate}
\end{prop}
\begin{proof}
The period map $\Prd_T$ preserves the orbifold sturctures if and only if for any $\omega\in \DD_T-\calH_T$ and an associated $Z(F)\in \PP\calV_T$, $\Stab_{Z(F)} \cong \Stab_\omega$.
By Corollary \ref{cor: Lambda_F^iota = P_F^iota}, $\iota_{F_0}$ acts as $-1$ on $Q_{F_0}$, this is equivalent to say that the map
\begin{equation*}
    \Aut(X_F,H_F)\cong \Or(\Lambda_F,H_F,\Delta_F,H^{2,0}(X_F)) \to \{ g|_{Q_{F_0}} \,\big|\, g\in \Or(\Lambda_{F_0},H_{F_0},\Delta_{F_0},\omega) \}
\end{equation*}
is an isomorphism.
Denote by $G_{Z(F)}\coloneqq \Or(\Lambda_F,H_F,\Delta_F,H^{2,0}(X_F))$ and $G_\omega\coloneqq \{ g|_{Q_{F_0}} \,\big|\, g\in \Or(\Lambda_{F_0},H_{F_0},\Delta_{F_0},\omega) \}$.

By definition, there is a surjective group homomorphism $p\colon \Or(\Lambda_{F_0},H_{F_0},\Delta_{F_0})\to \Gamma_T$ defined by restriction.
The stabilizer $G_{Z(F)}$ can be naturally embedded into $\Or(\Lambda_{F_0},H_{F_0},\Delta_{F_0})$ by the conjugation of $\varphi$.
Via this embedding, $G_{Z(F)}$ is isomorphic to $p^{-1}(G_\omega)$.
Hence $G_{Z(F)} \cong G_\omega$ is equivalent to $\Or(\Lambda_{F_0},H_{F_0},\Delta_{F_0}) \cong \Gamma_T$.
The isomorphism $\Or(\Lambda_{F_0},H_{F_0},\Delta_{F_0}) \cong \Gamma_T$ is equivalent to the condition that the identity element (hence every element) in $\Gamma_T$ admits a unique lift in $\Or(\Lambda_{F_0},H_{F_0},\Delta_{F_0})$, which means that $\Or(P_{F_0},H_{F_0},\Delta_{F_0})$ acts faithfully on $A_{P_{F_0}}$.
\end{proof}

\subsection{Nodal Singular Types}
As an application, we study the nodal singular types, i.e. the singular types with only nodes.
\begin{prop}
\label{prop: orbi str of nodal types}
    The orbifold structures are preserved in all nodal singular types.
\end{prop}
The proposition will be concluded from the following Lemma \ref{lem: orbi str of irred nodal types} and Lemma \ref{lem: orbi str of reducible types}.
\begin{lem}
\label{lem: orbi str of irred nodal types}
    For every irreducible nodal singular type, the orbifold structure is preserved.
\end{lem}
\begin{proof}
    Let $r_1,\cdots,r_m$ be the base of $L=A_1^{\oplus m}$.
    We have $\Or(P,H,\Delta)\cong\Aut(\Delta)\cong\mathfrak{S}_m$ and $A_P\cong\ZZ/2\times(\ZZ/2)^m$.
    Moreover, $\Or(P,H,\Delta)$ acts on $A_P$ by permuting the $m$ generators of the later $m$ $\ZZ/2$-factors.
    This action is faithful.
\end{proof}

\begin{lem}
\label{lem: 0 or >=5}
    Fix an arbitrarily nodal singular type.
    Let $r_1,\cdots,r_m$ be the base of $L=A_1^{\oplus m}$.
    For any $v\in P$ such that $v=\mu H+\sum\limits_i \lambda_i r_i$, either all $\lambda_i$ are integers, or there exist at least five of $\lambda_i$ belonging to ${1\over2}\ZZ\bs\ZZ$.  
\end{lem}
\begin{proof}
    For any irreducible nodal singular type, the statement is followed by $P=\langle H\rangle\oplus L$.
    
    Consider a reducible nodal singular type. 
    Denote by $C_1,\cdots,C_l$ the irreducible components of a sextic curve of the given singular type.
    In reducible cases, the statement is equivalent to saying that for any $C_{i_1},\cdots,C_{i_k}$ with $1\le k<l$, $C_{i_1}+\cdots+C_{i_k}$ has at least five coefficients relative to $r_1,\cdots,r_m$ belonging to ${1\over2}\ZZ\bs\ZZ$. 

    Each $C_i$ can be expressed as 
    \begin{equation*}
        C_i = {1\over2}(d_i H + \sum_{j\in I_i}r_j),
    \end{equation*}
    where $d_i=\deg C_i$, $I_i$ corresponds to the set of nodes arising from the intersection of $C_i$ with the union of other irreducible components.
    It follows that $C_{i_1}+\cdots+C_{i_k}$ satisfies 
    \begin{equation*}
        C_{i_1}+\cdots+C_{i_k} \equiv {1\over2}(\sum_{n=1}^k d_{i_n})H + {1\over2}\sum_{j\in I_{i_1,\cdots,i_k}} r_j \mod \langle H\rangle\oplus L,
    \end{equation*}
    where $I_{i_1,\cdots,i_k}$ corresponds to the set of nodes arising from the intersection of $C_{i_1}\cup\cdots\cup C_{i_k}$ with the union of other irreducible components.
    By B\'ezout theorem, the cardinality of $I_{i_1,\cdots,i_k}$ is at least five, thus the lemma follows.
\end{proof}

\begin{lem}
\label{lem: orbi str of reducible types}
    For every reducible nodal singular type, the orbifold structure is preserved.
\end{lem}
\begin{proof}
    Let $g\in\Or(P,H,\Delta)$ such that $g^*$ acts on $A_P$ trivially.
    By Proposition \ref{prop: orbifold criterion}, it is enough to show $g=\Id$, namely, $g$ preserves all $r_i, 1\le i\le m$. 
    For each $r_i$ arising from a self-intersection node of an irreducible component, $r_i$ is preserved for the same reason as in Lemma \ref{lem: orbi str of irred nodal types}.
    Thus it is enough to show that each $r_i$ arising from an intersection of two irreducible components is preserved.
    Let $C,C'$ be two different irreducible components.
    Denote by $J\coloneqq C\cap C'$.
    
    If $|J|=1$, then both of $C$ and $C'$ must be two lines.
    Denote by $r$ the associated exceptional curve.
    Consider ${1\over2}(H+r)\in P\otimes\QQ$, which is actually an element in $P^\vee$.
    Since $g^*=\Id$ on $A_P$, we have $g^*[{1\over2}(H+r)]=[{1\over2}(H+r)]=[{1\over2}(H+g r)]$ in $A_P$, then $[{1\over2}(r-g r)]=0$ in $A_P$.
    If $g^*r\ne r$, then ${1\over2}(r-g r)\ne 0$ lies in $P$, which contradicts to Lemma \ref{lem: 0 or >=5}.

    Consider the case $|J|\ge 2$.
    We first assume that for any $p_i,p_j\in J$, the associated exceptional curves $r_i, r_j$ not form a $g$-orbit in $P$.
    Consider ${1\over2}(r_i+r_j)\in P^\vee$.
    We have $g^*[{1\over2}(r_j+r_j)]=[{1\over2}(r_j+r_j)]=[{1\over2}(g r_i+g r_j)]$ in $A_P$.
    If $g r_i\ne r_i$, then ${1\over2}(r_i+r_j-g r_i-g r_j)\ne 0$ belongs to $P$, which contradicts to Lemma \ref{lem: 0 or >=5}.

    Otherwise, there exist $p_i,p_j\in J$ such that $g r_i=r_j, g r_j=r_i$.
    If $|J|\ge 3$, then there exists a $p_k\in J\bs\{p_i,p_j\}$.
    Consider ${1\over2}(r_i+r_k)\in P^\vee$.
    Similarly, we have ${1\over2}(r_i-r_j+r_k-g r_k)\ne 0$ belonging to $P$. Contradiction.

    Otherwise, $|J|=2$, then $C,C'$ must be a line and a conic.
    Then there exist only three possible cases: one conic plus four lines, two conics plus two lines and one cubic plus one conic plus one line.
    Note that the third case actually contains more than one singular types, since the cubic could be smooth or nodal. 
    But this is irrelevant to our discussion.

    We can assume that $C$ is a line and $C'$ is a conic.
    Let $C\cap C'=\{p_i,p_j\}$ and let $p_k$ be a node arising from the intersection of $C'$ with an irreducible component of odd degree (which always exists).
    It is direct to verify that ${1\over2}(H+r_i+r_k)\in P^\vee$.
    Then we conclude that ${1\over2}(r_i-r_j+r_k-g^*r_k)\ne 0$ belongs to $P$.
    Contradiction.
\end{proof}

\begin{rmk}
    It is interesting to know whether the orbifold structures are preserved for other singular types. This question is still open.
\end{rmk}


\section{Examples and Applications}
\label{sec: examples and applications}
The moduli and period map for nodal singular types have been well studied, see \cite{yu2023moduli}. 
In this section, we describe some further examples as applications. 

\subsection{A Quintic Curve with a Line}
We consider sextic curves that split into a smooth quintic curve and a line.
Let $V$ denote a complex vector space of dimension $3$.

\subsubsection{Five Nodes}
Suppose there is no tangent point. Denote the singular type by $T_0$. This case is first studied by Laza \cite{laza2009deformations}, and is related to the study of the minimal-elliptic surface singularity $N_{16}$.
This case is also included in \cite{yu2023moduli}.

In this case, $\calV_{T_0}$ consists of pairs $(F_5,F_1)\in \Sym^5 V^*\times V^*$ such that $Z(F_5)$ and $Z(F_1)$ have five different intersection points.
It is direct $\dim \calM_{T_0}=14$, which coincides with the expected dimension. The normalization $\widehat{\PP\calV}_{T_0}$ of $\PP\overline{\calV}_{T_0}$ is given by $\PP(\Sym^5 V^*) \times \PP(V^*)$, see \cite[\S 4.2]{yu2023moduli}.

For any fixed $F\in\calV_{T_0}$, we have $L_F\cong A_1^{\oplus 5}$.
Denote $\Delta_F$ by $\{ e_1,\cdots,e_5 \}$.
Corollary \ref{cor: if only A_1,D_2n,E_7,E_8} implies $P_F=P_F^{\iota}$. By Proposition \ref{prop: saturation of M_F}, we have:
\begin{prop}
    The lattice $P_F$ is generated by $\langle H\rangle\oplus L_F$ and ${H-e_1-\cdots-e_5\over2}$.
\end{prop}
\begin{rmk}
    For a generic $F\in\calM_{T_0}$, the K3 surface $X_F$ admits an elliptic fibration without sections (it admits a bi-section). This fibration has $1$ type $I_0^*$ fiber and $18$ nodal fibers. 
    The lattice $P_F$ is isomorphic to $U(2)\oplus D_4$. 
    The Picard lattice of the associated Jacobian elliptic K3 surface is $U\oplus D_4$, and the Mordell--Weil group is trivial. This example also appears in \cite[Table 1.1, No. 6]{belcastro2002picard} and \cite[Table 2]{clingher2024neron-severi}.
\end{rmk}
Recall that $Q_F\coloneqq P_F^\perp$, and $Q_F\cong U(2)\oplus U\oplus D_4\oplus E_8$.
The associated period domain is $\DD_{T_0}\coloneqq\DD(Q_F)$, which is a type IV domain of dimension $14$. 
The associated arithmetic group $\Gamma_{T_0}$ has the following expression
\begin{equation*}
    \Gamma_{T_0} = \{ g\in\Or(Q_F) \,|\, g \textup{ acts on } A_{Q_F} \textup{ through the } \mathfrak{S}_5\textup{-action on } A_{P_F}  \},
\end{equation*}
where $\mathfrak{S}_5$-action permutes $e_1,\cdots,e_5$.


The hyperplane arrangement $\calH_{T_0}^*$ is empty. Thus $\Prd_{T_0}$ extends to an isomorphism $\calM_{T_0}^{\mathrm{ADE}}\cong \Gamma_{T_0}\bs\DD_{T_0}$, see \cite[Theorem 4.1]{laza2009deformations} and \cite[Proposition 3.7]{yu2023moduli}\footnote{There is a small typo in \cite[Proposition 3.7]{yu2023moduli}: "$\Sigma_T$" should be revised as "$\overline{\Sigma}_T$", the closure of $\Sigma_T$ in the GIT compactification. This criterion holds for all ADE singular types.}.
The Looijenga compactification with respect to the empty hyperplane arrangement is exactly the Baily--Borel compactification.
Hence the period map $\Prd_{T_0}$ extends to an isomorphism $\overline{\calM}_{T_0}\cong \overline{\Gamma_{T_0}\bs\DD_{T_0}}^{\mathrm{bb}}$, also see \cite[Theorem 4.2]{laza2009deformations} and \cite[Proposition 4.3]{yu2023moduli}.

By Proposition \ref{prop: orbi str of nodal types}, the period map $\Prd_{T_0}$ induces an isomorphism of $\calM_{T_0}$ and $P\Gamma_{T_0}\bs(\DD_{T_0}-\calH_{T_0})$ as orbifolds.

\subsubsection{One Tacnode and Three Nodes}
Suppose that there is one tacnode and three nodes. Denote the singular type by $T_1$.
Let $\calV_{5}$ consist of $F_5\in\Sym^5 V^*$ such that $Z(F_5)$ is a smooth quintic curve.
In this case, $\PP\calV_{T_1}$ is a Zariski open subset of the universal family over $\calU_{5,1}$.

We construct $\widehat{\PP\calV}_{T_1}$ as follows. Define $G_1$ as $\{ 
(L,x)\in\PP(V^*)\times\PP(V) \,|\, x\in Z(L) \}$, which is a smooth $\PP^1$-bundle on $\PP(V^*)$. Define $G_{T_1}$ as 
\begin{equation*}
    \{ (L,x,F)\in G_1\times\PP(\Sym^5 V^*) \,|\, x\in Z(F),\, \mult_x(Z(F)\cdot Z(L))\ge 2 \}.
\end{equation*}
The space $G_{T_1}$ is a smooth $\PP^{18}$-bundle on $G_1$.
\begin{prop}
\label{prop: normalization for quintic plus 1 tangent}
    The normalization space $\widehat{\PP\calV}_{T_1}$ is isomorphic to $G_{T_1}$.
\end{prop}
\begin{proof}
    Let $\calU$ be the subspace of $\PP(V^*)\times\PP(V)\times\PP(\Sym^5 V^*)$ defined by 
    \begin{equation*}
        \{ (L,x,F) \,|\, x\in Z(L)\cap Z(F) \}.
    \end{equation*}
    The natural morphism $G_{T_1}\to \PP(\Sym^6 V^*)$ factors through the following morphisms
    \begin{equation*}
        G_{T_1} \to \calU \to \PP(V^*)\times\PP(\Sym^5 V^*) \to \PP(\Sym^6 V^*).
    \end{equation*}
    The first map is a natural inclusion, the second map is finite surjective and of degree $5$, and the last map is finite and generically injective. Notice that $G_{T_1}\to \PP(\Sym^6 V^*)$ is generically injective. Therefore, $G_{T_1}\to \PP(\Sym^6 V^*)$ is finite and generically injective.
    Its image contains $\PP\calV_{T_1}$ as a Zariski open subset, and thus equals $\PP\overline{\calV}_{T_1}$. Therefore, $G_{T_1}$ is the normalization of $\PP\overline{\calV}_{T_1}$.
\end{proof}

For any fixed $F\in\calV_{T_1}$, we have $L_F\cong A_3\oplus A_1^{\oplus 3}$.
Denote by $\{\epsilon_1,\epsilon_2,\epsilon_3\}$ the effective base of $A_3$ and by $e_1,e_2,e_3$ the effective bases of the three copies of $A_1$.
Then $\Delta_F$ is given by $\{ \epsilon_1,\epsilon_2,\epsilon_3;e_1;e_2;e_3 \}$.
\begin{lem}
\label{lem: 1 A_3 + 3 A_1: L_F^-iota = P_F^-iota}
    We have $L_F^{-\iota} = P_F^{-\iota}$.
\end{lem}
\begin{proof}
    By definition, $L_F^{-\iota}$ is generated by $\epsilon_1-\epsilon_3$, where $(\epsilon_1-\epsilon_3)^2=-4$.
    Since $({\epsilon_1-\epsilon_3\over2})^2=-1$ and $P_F$ is even, $L_F^{-\iota}$ does not admit any further finite-index extension in $P_F$, namely $L_F^{-\iota} = P_F^{-\iota}$.
\end{proof}
Recall that $M_F$ denotes the lattice generated by $\langle H_F\rangle\oplus L_F$ and $u\coloneqq{H-(2\epsilon_2+\epsilon_1+\epsilon_3)-e_1-e_2-e_3\over2}$.
\begin{prop}
\label{prop: 1 A_3 + 3 A_1: P=M}
    We have $P_F=M_F$. 
\end{prop}
\begin{proof} 
    It is direct to check that there exists a decomposition of $M_F$ as following
    \begin{equation*}
        M_F = \langle H_F-e_2,H_F-e_3 \rangle \oplus \langle e_1,u,\epsilon_2,\epsilon_1,\epsilon_3 \rangle \cong U(2)\oplus D_5.
    \end{equation*}
    Thus we have $A_{M_F}\cong (\ZZ/2)^2\times\ZZ/4$.
    By Lemma \ref{lem: 1 A_3 + 3 A_1: L_F^-iota = P_F^-iota} and Proposition \ref{prop: saturation of M_F}, the quotient $P_F\over{M_F^\iota\oplus M_F^{-\iota}}$ is isomorphic to a subgroup of $A_{M_F^\iota}$ and also a subgroup of $A_{M_F^{-\iota}}$.
    Since $A_{M_F^\iota}\cong(\ZZ/2)^4$ and $A_{M_F^{-\iota}}\cong\ZZ/4$, $P_F\over{M_F^\iota\oplus M_F^{-\iota}}$ is a subgroup of $\ZZ/2$.
    Moreover, we have ${M_F\over{M_F^\iota\oplus M_F^{-\iota}}}\cong\ZZ/2$.
    By definition, $M_F\subset P_F$, thus $M_F=P_F$.
\end{proof}




\begin{rmk}
    For a generic $F\in\calM_{T_1}$, the K3 surface $X_F$ admits an elliptic fibration $\pi_1$ with sections. This fibration has a type $I_0^*$ fiber and $18$ nodal fibers.
    The Mordell--Weil group of $\pi_1$ is of rank $1$ and torsion-free.
\end{rmk}
\begin{rmk}
    For a generic $F\in\calM_{T_1}$, $X_F$ also admits an elliptic fibration $\pi_2$ without sections (it admits a bi-section). This fibration has a type $I_1^*$ fiber and $17$ nodal fibers.
    The Jacobian elliptic surface of $\pi_2$ has Picard lattice isomorphic to $U\oplus D_5$. The Mordell--Weil group is trivial.
\end{rmk}

We have $Q_F\cong U\oplus U(2)\oplus A_3\oplus E_8$.
The period domain $\DD_{T_1}$ is a type IV domain of dimension $13$.
The arithmetic group $\Gamma_{T_1}$ can be expressed as
\begin{equation*}
    \Gamma_{T_1} = \{ g\in\Or(Q_F) \,|\, g \textup{ acts on } A_{Q_F} \textup{ through the } \mathfrak{S}_2\times\mathfrak{S}_3\textup{-action on } A_{P_F} \},
\end{equation*}
where $\mathfrak{S}_2$-action (respectively, $\mathfrak{S}_3$-action) permutes $\epsilon_1,\epsilon_3$ (respectively, $e_1,e_2,e_3$).
By \cite[Proposition 3.7]{yu2023moduli}, the hyperplane arrangement $\calH_{T_1}^*$ is empty.
The period map $\Prd_{T_1}$ extends to an isomorphism $\calM_{T_1}^{\mathrm{ADE}}\cong\Gamma_{T_1}\bs\DD_{T_1}$ and further an isomorphism $\overline{\calM}_{T_1}\cong \overline{\Gamma_{T_1}\bs\DD_{T_1}}^{\mathrm{bb}}$.



\subsubsection{Two Tacnodes and One Node}

Suppose there are two tangent points. Denote the singular type by $T=T_2$.
Due to a classical result by Schubert \cite{schubert1879kalkul}, a generic smooth degree $d$ plane curve has $d(d-2)(d^2-9)\over2$ bitangents.
Thus a generic smooth plane quintic curve has $120$ bitangents.
Denote by $\calM_5$ the moduli space of smooth quintic curves.
The natural projection $\calM_{T_2}\to \calM_5$ is a finite morphism of degree $120$.

We construct $\widehat{\PP\calV}_{T_2}$ as follows.
Define $\widetilde{G}_1$ as $\{ (L,x,y)\in\PP(V^*)\times\PP(V)\times\PP(V) \,|\, x,y\in Z(L) \}$, which is a smooth $\PP^1\times\PP^1$-bundle on $\PP(V^*)$.
The total space $\widetilde{G}_1$ admits an (fiberwise) involution $\tau$ defined by switching $x$ and $y$.
Define $G_1\coloneqq \widetilde{G}_1/\tau$. Since the fixed locus of $\tau$ has codimension $1$ in $\widetilde{G}_1$, $G_1$ is smooth. In fact, it is a $\PP^2$-bundle on $\PP(V^*)$. Denote by $[L,x,y]$ the image of $(L,x,y)\in\widetilde{G}_1$ in $G_1$.
Define $G_{T_2}$ as 
\begin{equation*}
    \{ ([L,x,y],F)\in G_1\times\PP(\Sym^5 V^*) \,|\, x,y\in Z(F),\, Z(F)\textup{ satisfies }(\star) \},
\end{equation*}
where the condition $(\star)$ means: when $x\ne y$, $\mult_x (Z(F)\cdot Z(L))\ge 2$ and $\mult_y (Z(F)\cdot Z(L))\ge 2$; when $x=y$, $\mult_x(Z(F)\cdot Z(L))\ge 4$. 
It is a smooth $\PP^{16}$-bundle on $G_1$. Similarly to Proposition \ref{prop: normalization for quintic plus 1 tangent}, we can describe the normalization of $\PP\overline{\calV}_{T_2}$ as follows.
\begin{prop}
    The normalization space $\widehat{\PP\calV}_{T_2}$ is isomorphic to $G_{T_2}$.
\end{prop}

For any fixed $F\in\calM_{T_2}$, we have $L_F\cong A_3^{\oplus 2}\oplus A_1$.
Denote by $(\epsilon_i)_i$ and $(\delta_i)_i$ the effective bases of two copies of $A_3$ respectively, and by $e$ the effective base of $A_1$.
Then $\Delta_F$ is given by $\{ \epsilon_1,\epsilon_2,\epsilon_3;\delta_1,\delta_2,\delta_3;e \}$.
\begin{lem}
\label{lem: 2 A_3 + 1 A_1: L^-iota = P^-iota}
    We have $L_F^{-\iota}=P_F^{-\iota}$.
\end{lem}
\begin{proof}
    By definition, $L_F^{-\iota}$ is generated by $\epsilon_1-\epsilon_3$ and $\delta_1-\delta_3$.
    The only possible value for $(a,b)$ $(0\le a,b\le 3)$ that results in $a(\epsilon_1-\epsilon_3)+b(\delta_1-\delta_3)\over4$ having an even self-intersection number is $(a,b)=(2,2)$.
    However, $(\epsilon_1-\epsilon_3)+(\delta_1-\delta_3)\over2$ is a root that is orthogonal to $H_F$, so it does not belong to $P_F$. 
    Thus $L_F^{-\iota}$ does not admit any further finite-index extension in $P_F$, namely $L_F^{-\iota} = P_F^{-\iota}$.
\end{proof}
In this case, $M_F$ is generated by $\langle H_F\rangle\oplus L_F$ and $u\coloneqq {H_F-(2\epsilon_2+\epsilon_1+\epsilon_3)-(2\delta_2+\delta_1+\delta_3)-e\over2}$.
\begin{prop}
    We have $M_F \cong \langle -4 \rangle\oplus U\oplus D_5$.
\end{prop}
\begin{proof}
    It is easy to check $A_{M_F}\cong(\ZZ/4)^2$, which is generated by ${H\over2}+{3\epsilon_1+2\epsilon_2+\epsilon_3\over4}$ and ${e\over2}+{3\delta_1+2\delta_2+\delta_3\over4}$.
    By \cite[Chapter 15, \S 7.4]{conway1999sphere}, we have $A_{M_F}\cong 4^{-2}_{2}$ (Conway--Sloane symbol).
    By \cite[Cor 1.13.3]{nikulin1979integer}, there exists a unique $N$ (up to isomorphism) of signature $(1,7)$ such that $A_N\cong 4^{-2}_{2}$.
    It is direct to verify that the discriminant form of $\langle -4 \rangle\oplus U\oplus D_5$ is isomorphic to $A_N\cong 4^{-2}_{2}$.
\end{proof}
\begin{prop}
    We have $P_F=M_F$. 
\end{prop}
\begin{proof}
    It is clear that $A_{M_F^\iota}\cong(\ZZ/2)^4$. 
    By Lemma \ref{lem: 2 A_3 + 1 A_1: L^-iota = P^-iota}, $A_{M_F^{-\iota}}\cong(\ZZ/4)^2$.
    By a parallel argument in the proof of Proposition \ref{prop: 1 A_3 + 3 A_1: P=M}, we can conclude $P_F=M_F$.
\end{proof}

\begin{rmk}
    For a generic $F\in\calV_T$, the K3 surface $X_F$ admits an elliptic fibration $\pi_1$ with sections.
    There exists a type $I_1^*$ fiber and $17$ nodal fibers.
    The Mordell--Weil group is of rank one and torsion-free.
\end{rmk}

We have $Q_F\cong \langle 4\rangle\oplus A_3\oplus U\oplus E_8$.
The arithmetic group $\Gamma_{T_2}$ can be described as
\begin{equation*}
    \Gamma_{T_2} = \{ g\in\Or(Q_F) \,|\, g \textup{ acts on } A_{Q_F} \textup{ through the } (\mathfrak{S}_2)^2\times\mathfrak{S}_2\textup{-action on } A_{P_F} \},
\end{equation*}
where the first two $\mathfrak{S}_2$-actions permute $\epsilon_1,\epsilon_3$ and $\delta_1,\delta_3$ respectively, and the last $\mathfrak{S}_2$-action permutes the two sets $\{\epsilon_i\}$ and $\{\delta_i\}$. 
Analogously to the previous cases, the period map $\Prd_{T_2}$ extends first to an isomorphism $\calM_{T_2}^{\mathrm{ADE}} \cong \Gamma_{T_2} \bs \DD_{T_2}$ and then further to $\overline{\calM}_{T_2} \cong \overline{\Gamma_{T_2} \bs \DD_{T_2}}^{\mathrm{bb}}$. 


\subsection{A Quartic Curve with Two Bitangents}
The nodal case of quartic curves with two lines is discussed in \cite{gallardo2018compactifications}.
We consider the singular type $T$ when the two lines are both bitangents.
Each smooth quartic curve admits $28$ bitangents.
Let $\calM_4$ denote the moduli space of quartic curves.
The natural projection $\calM_T\to \calM_4$ is a finite morphism of degree $\binom{28}{2}$.


For any fixed $F\in\calV_{T}$, we have $L_F\cong A_1\oplus A_3^{\oplus 4}$.
Denote by $\Delta\coloneqq \{ e;\alpha_1,\alpha_2,\alpha_3;\beta_1,\beta_2,\beta_3;\\ \epsilon_1,\epsilon_2,\epsilon_3;\delta_1,\delta_2,\delta_3 \}$.
In this case, $M_F$ is generated by $\langle H\rangle \oplus L$, $u\coloneqq{H-e-(2\alpha_2+\alpha_1+\alpha_3)-(2\beta_2+\beta_1+\beta_3)\over2}$ and $v\coloneqq{H-e-(2\epsilon_2+\epsilon_1+\epsilon_3)-(2\delta_2+\delta_1+\delta_3)\over2}$. 
\begin{lem}
    We have $M_F^{-\iota}=P_F^{-\iota}$.
\end{lem}
\begin{proof}
    Let $w\coloneqq{a(\alpha_1-\alpha_3)+b(\beta_1-\beta_3)+c(\epsilon_1-\epsilon_3)+d(\delta_1-\delta_3)\over4}$ be an element in $(A_3^{\oplus4})^\vee$, where $0\le a,b,c,d\le 3$.
    If $[w]$ is an isotropic element in $A_{A_3^{\oplus4}}$, then we have $8\,|\,a^2+b^2+c^2+d^2$. 
    Then either $(a,b,c,d)$ has two entries equal to $2$ and two to $0$, or $(a,b,c,d)$ is equal to $(2,2,2,2)$. 
    In the former case, all the entries correspond to roots orthogonal to $H_F$, none of which lies in $P_F$.
    The element corresponding to $(2,2,2,2)$ is contained in $M_F^{-\iota}$. 
\end{proof}
\begin{ques}
    Do we have $P_F=M_F$?
\end{ques}
\begin{rmk}
    Since $A_{M_F^{\iota}}\cong(\ZZ/2)^6$ and $A_{M_F^{-\iota}}\cong(\ZZ/4)^2\times(\ZZ/2)^2$, there may exist a finite-index extension $M_F^{\iota}\oplus M_F^{-\iota}\hookrightarrow P$ such that ${P\over M_F^{\iota}\oplus M_F^{-\iota}}\cong(\ZZ/2)^4$.
    Notice that ${M_F\over M_F^{\iota}\oplus M_F^{-\iota}}\cong(\ZZ/2)^3$.
    There may not be an easy way to decide whether $P_F=M_F$.
\end{rmk}

\begin{rmk}
    For a generic $F\in\calM_T$, $X_F$ admits an elliptic fibration $\pi_1$ with sections.
    There exists a type $I_1^*$ fiber, $2$ type $I_4$ fibers and $15$ nodal fibers.
\end{rmk}
\begin{rmk}
    For a generic $F\in\calM_T$, $X_F$ admits another elliptic fibration without sections (it admits a bi-section).
    There exist $2$ type $I_2^*$ fibers and $10$ nodal fibers.
\end{rmk}

The arithmetic group $\Gamma_{T}$ can be characterized as
\begin{equation*}
    \Gamma_{T} = \{ g\in\Or(Q_F) \,|\, g \textup{ acts on } A_{Q_F} \textup{ through the } (\mathfrak{S}_2)^4\times\mathfrak{S}_4\textup{-action on } A_{P_F} \},
\end{equation*}
where the four $\mathfrak{S}_2$-actions permutes $\alpha_1,\alpha_3$; $\beta_1,\beta_3$; $\epsilon_1,\epsilon_3$ and $\delta_1,\delta_3$ respectively, and the $\mathfrak{S}_4$-action permutes the four sets $\{\alpha_i\}$, $\{\beta_i\}$, $\{\epsilon_i\}$ and $\{\delta_i\}$.
Following the same argument, $\Prd_{T}$ extends to an isomorphism $\calM_{T}^{\mathrm{ADE}} \cong \Gamma_{T} \bs \DD_{T}$ and further to $\overline{\calM}_{T} \cong \overline{\Gamma_{T} \bs \DD_{T}}^{\mathrm{bb}}$.

The moduli space of quartic curves admits an occult period map to an arithmetic ball quotient dimension $6$, see \cite{kondo2011moduli}.
    A smooth quartic curve is a Riemann surface of genus three.
    Its Jacobian is a principal polarized abelian variety of dimension three. This leads to a birational equivalence between the moduli space of quartic curves and an arithmetic quotient of the Siegel upper half-space $\mathcal{S}_3$ of degree $3$.
    It would be interesting to study more relations among the three arithmetic quotients of different types and their compactifications.




\subsection{Zariski Pairs}
\label{subsec: Zariski pair}
These examples had been discussed in \cite[\S 3]{yang2012discriminantal}.
We present a construction of them from a more geometric viewpoint.



A pair of reduced plane curves of the same degree, which have the same combinatorial type of singularities but are not equisingular deformation equivalent, is usually called \emph{Zariski pairs}.
Zariski showed that the space of irreducible sextic curves with six cusps as the only singularities form at least two irreducible components, which provided the first example, see \cite{zariski1931irregularity}. 
Degtyarev showed that they actually form exactly two irreducible components, see \cite[Theorem 5.3.2]{degtyarev2008deformations}.
Both of the two irreducible components are smooth of the same expected dimension. 

Denote by $T_1,T_2$ the two corresponding singular types, where $T_2$ corresponds to sextic curves whose six cusps lie on a conic.

Let $C$ be a sextic of $T_2$-type.
Denote by $D$ for the conic containing the six cusps $p_1,\cdots,p_6$ of $C$.
By B\'ezout theorem, $D\cap C = \{ p_1,\cdots,p_6 \}$ and the intersection multiplicity of each $p_i$ is two.
Let $f\colon S\to \PP^2$ be the blowup of $\{ p_1,\cdots,p_6 \}$ on $\PP^2$, and denote by $E_1,\cdots,E_6$ the exceptional curves.
We still denote by $C$ (respectively, $D$) the strict transform of $C$ (respectively, $D$).
Then $C$ and $D$ are disjoint in $S$.
For each $1\le i\le 6$, we have $E_i^2=-1$ and $E_i\cdot D=1$.

Let $\pi\colon X\to S$ be the double cover branched over $C$.
Then $\pi^{-1}(D)$ is a disjoint union of $D_1$ and $D_2$, where $D_1,D_2$ are both isomorphic to $D$.
Moreover, we have $\pi^* E_i = \widetilde{E}_i + \widetilde{F}_i$ and $\pi^* C = 2 \widetilde{C}$, where each $(\widetilde{E}_i,\widetilde{F}_i)$ forms a $A_2$-type Dynkin diagram, $\widetilde{C}$ is reduced and $\widetilde{E}_i\cdot \widetilde{C} = \widetilde{F}_i\cdot \widetilde{C} = 1$.
Let $H$ be the pullback of the hyperplane class on $\PP^2$.
We have $H^2=2$ and $H\cdot D_i=2$.
It is direct to verify that 
\begin{equation*}
    D_1 = H - \frac{1}{3}\sum_{i=1}^6 (2 \widetilde{E}_i + \widetilde{F}_i),\quad D_2 = H - \frac{1}{3}\sum_{i=1}^6 (\widetilde{E}_i + 2 \widetilde{F}_i).
\end{equation*}
It is clear that neither $D_1$ nor $D_2$ lies in $\left\langle H\right\rangle \oplus L$.

Due to the classification by Zariski and Degtyarev, we have the following:
\begin{prop}
\label{proposition: picard zariski pair}
\begin{enumerate}[(i)]
    \item For the $T_1$-type, the generic Picard group $P_{T_1}$ of the associated K3 surfaces is generated by $H$ and $L$.
    \item For the $T_2$-type, the generic Picard group $P_{T_2}$ of the associated K3 surfaces is generated by $H$, $L$ and $D_1$.
\end{enumerate}
\end{prop}
The singular type $T_2$ provides an example for $P\ne M$.

We have $Q_{T_1}=P_{T_1}^\perp\cong \langle -2\rangle\oplus U(3)^{\oplus 2}\oplus A_2^{\oplus 2}$, $Q_{T_2}=P_{T_2}^\perp\cong \langle -2\rangle\oplus U\oplus U(3)\oplus A_2^{\oplus 2}$.
Both of the period domains $\DD_{T_1}$ and $\DD_{T_2}$ are type IV domains of dimension $7$.
The arithmetic group $\Gamma_{T_1}$ can be written in the form of
\begin{equation*}
    \Gamma_{T_1} = \{ g\in\Or(Q_{T_1}) \,|\, g \textup{ acts on } A_{Q_{T_1}} \textup{ through the } (\mathfrak{S}_2)^6\times\mathfrak{S}_6\textup{-action on } A_{P_{T_1}} \}.
\end{equation*}
Similar for the arithmetic group $\Gamma_{T_2}$,
\begin{equation*}
    \Gamma_{T_2} = \{ g\in\Or(Q_{T_2}) \,|\, g \textup{ acts on } A_{Q_{T_2}} \textup{ through the } (\mathfrak{S}_2)^4\times\mathfrak{S}_4\textup{-action on } A_{P_{T_2}} \}.
\end{equation*}
For $i=1,2$, the period map $\Prd_{T_i}$ extends to an isomorphism $\calM_{T_i}^{\mathrm{ADE}} \cong \Gamma_{T_i} \bs \DD_{T_i}$ and further to $\overline{\calM}_{T_i} \cong \overline{\Gamma_{T_i} \bs \DD_{T_i}}^{\mathrm{bb}}$.

\bibliography{ref}
\bibliographystyle{alpha}
\Addresses

\end{document}